\documentclass[10pt,twoside]{article}
\usepackage[utf8]{inputenc}
\usepackage[english]{babel}
\usepackage{fancyhdr}
\usepackage{amssymb,amsmath,amsthm}
\usepackage{blindtext}
\usepackage{hyperref}
\usepackage{caption}
\usepackage[square,numbers]{natbib}

\usepackage{lipsum}
\usepackage{amsfonts}
\usepackage{graphicx}
\usepackage{epstopdf}
\usepackage{algorithmic}
\usepackage{algorithm}
\usepackage{mathtools}
\usepackage[inline]{enumitem}
\usepackage{xcolor}
\usepackage{subfig}
\usepackage{placeins}
\usepackage{bigstrut}
\usepackage{dcolumn}

\def\R{\mathbb R}
\def\conv{\text{conv}}

\DeclarePairedDelimiter{\norm}{\lVert}{\rVert} % norm

\newcommand{\ORD}{\texttt{ORD}}
\newcommand{\DFSIMPL}{\texttt{DF-SIMPLEX}}
\newcommand{\LINCOA}{\texttt{LINCOA}}
\newcommand{\NOMAD}{\texttt{NOMAD}}
\newcommand{\PSWARM}{\texttt{PSWARM}}
\newcommand{\SDPEN}{\texttt{SDPEN}}

\newtheorem{definition}{Definition}
\newtheorem{lemma}{Lemma}
\newtheorem{proposition}{Proposition}
\newtheorem{theorem}{Theorem}
\newtheorem{corollary}{Corollary}
\newtheorem{remark}{Remark}

\hypersetup{
    colorlinks=false,
    urlcolor=black,
    pdfborder={0 0 0}
}

\begin{document}
\thispagestyle{plain}

\setcounter{page}{1}

{\centering
%-----------------------%
% INSERT HERE THE TITLE %
%-----------------------%
{\LARGE \bfseries A derivative-free method for structured optimization problems\par}

\bigskip\bigskip
%-------------------------%
% INSERT HERE THE AUTHORS %
%-------------------------%
Andrea Cristofari$^*$, Francesco Rinaldi$^*$
\bigskip

}

%----------------------------------%
% INSERT HERE AUTHORS' INFORMATION %
%----------------------------------%
\begin{center}
\small{\noindent$^*$Department of Mathematics ``Tullio Levi-Civita'' \\
University of Padua \\
Via Trieste, 63, 35121 Padua, Italy \\
E-mail: \texttt{andrea.cristofari@unipd.it}, \texttt{rinaldi@math.unipd.it} \\
}
\end{center}

\bigskip\par\bigskip\par
\noindent \textbf{Abstract.}
Structured optimization problems are  ubiquitous in fields like data science and engineering. The goal in  structured optimization is using a prescribed set of points, called atoms, to build up a solution that minimizes or maximizes a given function.  In the present paper, we want to minimize a black-box function over the convex hull of a given set of atoms, a problem that can be used to model a number of real-world applications. We focus on problems whose solutions are sparse, i.e., solutions that can be obtained as a proper convex combination of just a few atoms in the set, and propose a suitable derivative-free inner approximation approach that nicely exploits the structure of the given problem.
This enables us to properly handle the dimensionality issues usually connected with derivative-free algorithms,  thus getting a method that scales well in terms of both the dimension of the problem and the number of atoms.
We analyze global convergence to  stationary points. Moreover, we
show that, under suitable assumptions, the proposed algorithm identifies a specific subset of atoms with zero weight in the final solution after finitely many iterations.
Finally, we report numerical results showing the effectiveness of the proposed method.


\bigskip\par
\noindent \textbf{MSC2000 subject classifications.} 90C06. 90C30. 90C56.

\section{Introduction}
In this paper, we consider an optimization problem of the type
\begin{equation}\label{prob}
\tag{$P0$}
\min_{x\in \mathcal{M}} f(x),
\end{equation}
where $\mathcal{M}$ is the convex hull of a finite set of points $\mathcal{A}=\{a_1,\ldots,a_m\} \subset \mathbb{R}^n$ called \emph{atoms}
(some of them might not be extreme points of $\mathcal{M}$)
and $f:\R^n \to \R$ is a continuously differentiable function.
We further assume that first-order information related to the objective function is unavailable or impractical to
obtain (e.g., functions are expensive to evaluate or somewhat noisy).
Since any point $x \in \mathcal{M}$ can be written as a convex combination of the atoms in $\mathcal{A}$,
Problem \eqref{prob} can be equivalently reformulated considering the
simplicial representation of the feasible set:
\begin{equation}\label{prob_sim0}
\tag{$P1$}
\min_{w\in \Delta_{m-1}} f(Aw),
\end{equation}
where $A=\begin{bmatrix}a_1 & \dots & a_m \end{bmatrix}\in \R^{n\times m}$ and $\Delta_{m-1}=\{w\in \R^m \colon e^T w=1,\ w\ge 0\}$,
with $e$ being the vector made of all ones.
Thus, each variable $w_i$ gives the weight of the $i$-th
atom in the convex combination.

We are particularly interested in instances of Problem \eqref{prob_sim0} that admit a sparse solution, i.e., instances whose solutions can be obtained as a proper convex combination of a small subset of atoms.

This  occurs, e.g., when $m \gg n$ (as a consequence of Carath\'eodory's theorem~\cite{caratheodory1907variabilitatsbereich}). We would like to notice that  this is not the only case that gives sparse solutions. We can have polytopes with ${\cal O}(n)$ vertices that, thanks to their structure, can induce sparsity anyway. A classic example is the $\ell_1$ ball~\cite{bach2011convex}.

This black-box structured optimization problem is somehow related to sparse \emph{atomic decomposition}  (see, e.g., \cite{chandrasekaran2012convex, fan2019polar} and references therein). In such a context the atomic structure
can be exploited when developing tailored  solvers for the problem.

There exists a significant number of real-world applications that fits our mathematical model. Interesting examples include, among others, black-box adversarial attacks on deep neural networks with $\ell_1$ or $\ell_\infty$ bounded perturbations (see, e.g., \cite{brendel2017decision,chen2017zoo,ilyas2018black} and references therein), and  reinforcement learning
(see, e.g., \cite{kaelbling1996reinforcement,ng2000pegasus} and references therein) with constrained policies.

In principle, Problem~\eqref{prob_sim0} can be tackled by any linearly constrained derivative-free optimization algorithm.
A large number of those methods %, which in principle might be considered  to tackle the reformulation given in~\eqref{prob_sim0},
are available in the literature. Nice overviews can be found in, e.g., \cite{audet2017derivative,conn:2009,kolda2003optimization,larson2019derivative}. An important class of methods is represented by direct-search schemes (see, e.g., \cite{kolda2003optimization} for further details). Those approaches explore the objective function along suitably chosen sets of directions that
somehow take into account the shape of the feasible region around the current iterate, and usually are
given by the positive generators of an approximate tangent cone related to nearby active constraints
\cite{kolda2007stationarity,lewis2000pattern}.
The chosen directions both guarantee feasibility and allow a decrease in the objective function value, when a sufficiently small stepsize is taken. Line search techniques can also be used to better explore the search directions \cite{lucidi2002objective}.
Moreover, conditions for the  active-set identification are described in~\cite{lewis2010active}.
% When dealing with linearly dependent constraints,
%  the computation of the positive generators is unfortunately hard and may require the use of enumerative techniques (see, e.g., \cite{abramson2008pattern,lewis2007implementing}), a drawback that might limit the practical use of those methods.

Another approach for the linearly constrained setting is proposed in~\cite{gratton2019direct}, where the authors introduce
the notions of deterministic and probabilistic feasible descent (they basically consider the projection of the negative gradient
on an approximate tangent cone identified by nearby active constraints). For the
deterministic case, a complexity bound for direct search (with sufficient
decrease) is given. They further prove global convergence with probability $1$ when using direct search based on probabilistic
feasible descent, and derive a complexity
bound with high probability.

% In order to handle linear equality constraints, we might use some specific approaches that include constraint removal by change of variables  \cite{audet2015linear} and exploration of the null space described by the constraints \cite{lewis2007implementing}. Those techniques might give a problem that is usually harder to solve (due to the reduced separability of the problem itself).

The use of global optimization strategies combined with direct-search approaches for linearly constrained problems has been investigated in \cite{diouane2015globally,vaz2007particle,vaz2009pswarm}.

Model-based approaches (see, e.g., \cite{audet2017derivative,conn:2009}) can also be used for solving linearly constrained derivative-free optimization problems. In~\cite{powell2015fast}, Powell described
trust-region methods for quadratic models
with linear constraints, which are used in the LINCOA software~\cite{lincoa}, developed by the same author for derivative-free linearly constrained optimization. Moreover, an extension of Powell's NEWUOA algorithm~\cite{powell2006newuoa, powell2008developments} to the linearly constrained case has been developed in~\cite{gumma2014derivative}.

Since the derivative-free strategies listed above do not exploit the peculiar structure of Problem \eqref{prob_sim0}, they might get stuck when the problem dimensions
increase.

Another way to deal with the original Problem \eqref{prob} is  by generating the facet-inducing halfspaces that describe the feasible set ${\cal M}$. In our case, the facet description could be obtained from the atom list by means of suitable facet enumeration strategies (see, e.g., \cite{avis1997good}). This might obviously help in case $m\gg n$ and the polytope has a specific structure. We anyway need to keep in mind that there exists a number of problems where using the facet description is not a viable option. A first example is when $\cal A$, is linear with respect to the problem dimension, but $\cal M$ does not have a polynomial description in terms of facet-inducing halfspaces. Another interesting example is given by problems where the inner description is not available and we only have an oracle that generates our atoms.
 Furthermore, since we consider instances whose solutions can be obtained using a very small number of atoms (usually much smaller than the dimension $n$), it would be better to exploit the vertex description when devising a new method.

We hence propose a new algorithmic scheme that tries
to take into account the features of the considered  problem, thus allowing us to solve large-scale instances. At each iteration, our approach performs three different steps:
\begin{enumerate}[label=(\roman*), leftmargin=*]
    \item it approximately solves a reduced problem whose feasible set is an inner description of $\cal M$ (given by the convex hull of a suitably chosen subset of atoms);
    \item it tries to refine the inner description of the feasible set by including new atoms;
    \item it tries to remove atoms by proper rules in order to keep the dimensions of the reduced problem small.
\end{enumerate}
More in detail, the approximate minimization of the reduced problem is
carried out by means of a tailored algorithm that combines the use of a specific set of sparse directions containing positive generators of the tangent cone at the
current iterate with a line search similar to those described in, e.g., \cite{lucidi2002derivative,lucidi2002global,lucidi2002objective}.
Furthermore, the addition/removal of new atoms guarantees an improvement of the objective function whenever we approximately solve the reduced problem.
Those key features enable us to prove the convergence of the method and, under suitable assumptions, the asymptotic finite identification of a specific subset of atoms with zero weight in the final solution.
This identification result has relevant implications on the computational side. The algorithm indeed keeps the reduced problem small enough along the iterations
when the final solution is sparse,
thus guaranteeing a significant objective function reduction even with a small budget of function evaluations.
%In the numerical result we report here, it is easy to see that the algorithm scales well with the problem dimension
%and is able to quickly find a good approximation of the face where a possible stationary point lies.
%This is pretty useful especially when dealing with instances whose solutions lie in a very small dimensional face of the polytope $\cal M$.

The proposed method is somehow related to inner approximation approaches (see, e.g., \cite{bertsekas2015convex} and references therein) for convex optimization problems.
Anyway, those methods cannot be directly applied to  the class of problems considered here due to the following reasons:
\begin{itemize}
    \item  they require assumptions on the objective functions that might be hard to verify in a DFO context;
    \item they normally use first/second order information to carry out the (approximate) minimization of the reduced problem and to select new atoms to be included in the inner description (see, e.g., \cite{hearn1987restricted, patriksson2013nonlinear}).
\end{itemize}
In our framework, we only require smoothness of the objective function and use zeroth order information (i.e., function evaluations)
to approximately minimize the reduced problem and to select a new atom.
To the best of our knowledge this is the first time that a complete theoretical and computational analysis of a derivative-free inner approximation approach is carried out.

% Let us introduce the notation used in the paper.
% Given a vector $x \in \R^n$, the $i$th component of $x$ is indicated by $x_i$.
% We denote by $e_i \in \R^n$ the $i$th vector of the canonical basis, i.e., the vector made of all zeros except for the $i$th component that is equal to $1$.
% The gradient of a function $f \colon \R^n \to \R$ is denoted by $\nabla f$ and $\nabla_i f$ denoted its $i$th component.
% The Euclidean norm of a vector $x \in \R^n$ is indicated by $\norm x$.

%-------------------------------------------------------------------------------
\section{A basic algorithm for minimization over the unit simplex}\label{sec:df_simplex}
In our framework, we need an inner solver for approximately minimize the objective function over a subset of atoms.
This motivates us to design a tailored approach for problems of the following form:
\begin{equation}\label{prob_sim}
\min_{y\in \Delta_{\bar m-1}} \varphi(y),
\end{equation}
where $\varphi \colon \R^{\bar m} \to \R$ is a continuously differentiable function.
The scheme of the method, that we named \DFSIMPL, is reported in Algorithm \ref{alg:basic}. It combines the use of a suitable set of sparse directions containing positive generators of the tangent cone at the
current iterate with a specific line search that guarantees feasibility.

We start by choosing a feasible point $y^0 \in \Delta_{\bar m-1}$ and some stepsizes $\hat \alpha^0_i,\ i=1,\dots, \bar m$
(note that we have a starting stepsize for each component $y_i$ of the solution).
At each iteration $k$, we select a variable index $j_k$ such that $y^k_{j_k}$ is ``sufficiently positive'' (see line 3 in Algorithm \ref{alg:basic}) and define the directions $d^k_i =\pm (e_i - e_{j_k})$, for all indices $i\neq j_k$, where with  $e_i \in \R^{\bar m}$ we denote from now on the $i$th vector of the canonical basis, i.e., the vector made of all zeros except for the $i$th component that is equal to $1$.
Search directions of this form are related to those used in the 2-coordinate descent method proposed in~\cite{cristofari2019almost}, with the difference that here, unlike in~\cite{cristofari2019almost}, first-order information is not available, and then, both $e_i - e_{j_k}$ and $e_{j_k}-e_i$ must be explored for all $i \ne j_k$.
Once these search directions are computed, for each of them we perform a line search to get a sufficient reduction in the objective function and we suitably update the values of the starting stepsizes $\hat \alpha^k_i$, $i=1,\ldots,\bar m$.
The line search procedure is reported in Algorithm~\ref{alg:ls}.
It is similar to those described in, e.g., \cite{lucidi2002derivative,lucidi2002global,lucidi2002objective}.
Notice that, in Algorithm~\ref{alg:ls}, we have $\bar \alpha = (z^k_i)_{j_k}$ at line 1 and $\bar \alpha = (z^k_i)_i$ at line 3.

\begin{algorithm}
\caption{\DFSIMPL}
\label{alg:basic}
\begin{algorithmic}
{\footnotesize
\par\vspace*{0.1cm}
\item[]\hspace*{-0.1truecm}$\,\,\,1$\hspace*{0.1truecm} Choose a point $y^0\in \Delta_{\bar m-1}$, $\tau \in (0,1]$, $\theta \in (0,1)$,
                                     $\gamma > 0$, $\delta \in (0,1)$ and  $\hat \alpha^0_1,\ldots,\hat \alpha^0_{\bar m} > 0$
\item[]\hspace*{-0.1truecm}$\,\,\,2$\hspace*{0.1truecm} For $k=0,1,\ldots$
\item[]\hspace*{-0.1truecm}$\,\,\,3$\hspace*{0.9truecm} Choose $j_k$ such that $\displaystyle{y^k_{j_k} \ge \tau \max_{i=1,\ldots,\bar m} y^k_i}$
                                   and let $\alpha^k_{j_k} = 0$
\item[]\hspace*{-0.1truecm}$\,\,\,4$\hspace*{0.9truecm} Set $z^k_1 = y^k$
\item[]\hspace*{-0.1truecm}$\,\,\,5$\hspace*{0.9truecm} For $i=1,\dots,\ \bar m$
\item[]\hspace*{-0.1truecm}$\,\,\,6$\hspace*{1.7truecm} If $(i\neq j_k)$ then
\item[]\hspace*{-0.1truecm}$\,\,\,7$\hspace*{2.5truecm} Set $\tilde d = e_i - e_{j_k}$ %or $d^k_i = e_{j_k} - e_i$
\item[]\hspace*{-0.1truecm}$\,\,\,8$\hspace*{2.5truecm} Compute $\alpha$ and $d$ by $\texttt{Line Search Procedure}(z^k_i,\tilde d,\hat \alpha^k_i,\gamma,\delta)$
\item[]\hspace*{-0.1truecm}$\,\,\,9$\hspace*{2.5truecm} If $\alpha = 0$, then set $\hat \alpha^{k+1}_i = \theta \hat \alpha^k_i$
\item[]\hspace*{-0.1truecm}$10$\hspace*{2.5truecm} else set $\hat \alpha^{k+1}_i = \alpha$
\item[]\hspace*{-0.1truecm}$11$\hspace*{1.7truecm} else set $\alpha=0$ and $d=0$
\item[]\hspace*{-0.1truecm}$12$\hspace*{1.7truecm} End if
\item[]\hspace*{-0.1truecm}$13$\hspace*{1.7truecm} Set $\alpha^k_i = \alpha$,   $d_i^k=d$ and $z^k_{i+1} = z^k_i + \alpha^k_i d^k_i$
\item[]\hspace*{-0.1truecm}$14$\hspace*{0.9truecm} End for
\item[]\hspace*{-0.1truecm}$15$\hspace*{0.9truecm} Let $\xi_i = \hat \alpha^{k+1}_i$, $i\in \{1,\ldots,\bar m\} \setminus\{j_k\}$, and $\xi_{j_k} = \hat \alpha^k_{j_k}$
\item[]\hspace*{-0.1truecm}$16$\hspace*{0.9truecm} Set $\displaystyle{\hat \alpha^{k+1}_{j_k} = \min_{i = 1,\ldots,\bar m} \xi_i}$
\item[]\hspace*{-0.1truecm}$17$\hspace*{0.9truecm} Set $y^{k+1} = z^k_{\bar m+1}$
\item[]\hspace*{-0.1truecm}$18$\hspace*{0.1truecm} End for
\par\vspace*{0.1cm}
}
\end{algorithmic}
\end{algorithm}

\begin{algorithm}
\caption{$\texttt{Line Search Procedure}(z,d,\hat \alpha,\gamma,\delta)$}
\label{alg:ls}
\begin{algorithmic}
{\footnotesize
\par\vspace*{0.1cm}
\item[]\hspace*{-0.1truecm}$\,\,\,1$\hspace*{0.1truecm} Compute the largest $\bar \alpha$ such that $z + \bar \alpha d \in \Delta_{\bar m-1}$  and  set $\alpha = \min\{\bar \alpha, \hat \alpha\}$
\item[]\hspace*{-0.1truecm}$\,\,\,2$\hspace*{0.1truecm} If $\alpha > 0$ and $\varphi(z+\alpha d) \le \varphi(z) - \gamma \alpha^2$, then go to line 6
\item[]\hspace*{-0.1truecm}$\,\,\,3$\hspace*{0.1truecm} Compute the largest $\bar \alpha$ such that $z - \bar \alpha d \in \Delta_{\bar m-1}$ and set $\alpha = \min\{\bar \alpha, \hat \alpha\}$
\item[]\hspace*{-0.1truecm}$\,\,\,4$\hspace*{0.1truecm} If $\alpha > 0$ and $\varphi(z-\alpha d) \le \varphi(z) - \gamma \alpha^2$, then set $d = -d$ and go to line 6
\item[]\hspace*{-0.1truecm}$\,\,\,5$\hspace*{0.1truecm} Set $\alpha = 0$ and go to line 10
\item[]\hspace*{-0.1truecm}$\,\,\,6$\hspace*{0.1truecm} Let $\beta = \min\{\bar \alpha, \alpha/\delta\}$
\item[]\hspace*{-0.1truecm}$\,\,\,7$\hspace*{0.1truecm} While ($\alpha < \bar \alpha$ and $\varphi(z+\beta d) \le \varphi(z) - \gamma \beta^2$)
\item[]\hspace*{-0.1truecm}$\,\,\,8$\hspace*{0.9truecm} Set $\alpha = \beta$ and $\beta = \min\{\bar \alpha, \alpha/\delta\}$
\item[]\hspace*{-0.1truecm}$\,\,\,9$\hspace*{0.1truecm} End while
\item[]\hspace*{-0.1truecm}$10$\hspace*{0.1truecm} Return $\alpha$, $d$
\par\vspace*{0.1cm}
}
\end{algorithmic}
\end{algorithm}

It should be noticed that, in practice, shuffling the
search directions used at each iteration $k$ can improve performances.
All the theoretical results that will be shown below can be easily adapted to that case.

\subsection{Theoretical analysis}

To analyze the theoretical properties of the algorithm,
let us first recall a stationarity condition for problem~\eqref{prob_sim}.
\begin{proposition}
A feasible point $y^*$ of Problem~\eqref{prob_sim} is stationary if and only if there exists $\lambda^* \in \R$ such that, for all $i = 1,\ldots,\bar m$,
\begin{equation}\label{kkt}
\nabla_i \varphi(y^*)
\begin{cases}
\ge \lambda^*, \quad & \text{if } y^*_i = 0, \\
= \lambda^*,   \quad & \text{if } y^*_i > 0.
\end{cases}
\end{equation}
\end{proposition}

We now show that the line search strategy embedded in \DFSIMPL\ always terminates in a finite number of steps.

\begin{proposition}
\texttt{Line Search Procedure} has finite termination.
\end{proposition}

\begin{proof}
We need to show that the while loop at lines 7--9 ends in a finite number of steps.
Arguing by contradiction, assume that this is not true. Then, within the while loop we generate a divergent monotonically increasing sequence of feasible stepsizes $\alpha$'s,
which contradicts the fact that $\Delta_{\bar m-1}$ is a bounded set.
\end{proof}

In the next proposition, we prove that the stepsizes $\alpha^k_i$ generated using our line search
go to zero. This is a
 standard technical result that will
be needed to show convergence of the algorithm.

\begin{proposition}\label{prop:lim_alpha}
Let $\{y^k\}$ be a sequence of points produced by \DFSIMPL. Then,
\[
\lim_{k \to \infty} \alpha^k_i = 0, \quad i = 1,\ldots,\bar m.
\]
\end{proposition}

\begin{proof}
For every fixed $i \in \{1,\ldots,\bar m\}$, we partition the iterations into two subsets $K'$ and $K''$ such that
\[
\alpha^k_i = 0 \Leftrightarrow k \in K' \quad \text{and} \quad \alpha^k_i \ne 0 \Leftrightarrow k \in K''.
\]
If $K''$ is a finite set, necessarily $\alpha^k_i = 0$ for all sufficiently large $k$ and the result trivially holds.
If $K''$ is an  infinite set, to obtain the desired result we need to show that
\begin{equation}\label{lim_alpha_2}
\lim_{\substack{k \to \infty \\ k \in K''}} \alpha^k_i = 0.
\end{equation}
By instructions of the algorithm, for all $k \in K''$ we have that
\[
\varphi(y^{k+1}) \le \varphi(z^k_{i+1}) \le \varphi(z^k_i) - \gamma (\alpha^k_i)^2 \le \varphi(y^k) - \gamma (\alpha^k_i)^2.
\]
Combining these inequalities with the fact that $\Delta_{\bar m-1}$ is a bounded set and $\varphi$ is continuous, it follows that $\{\varphi(y^k)\}$ converges
and, since $\varphi(y^k) - \varphi(y^{k+1}) \ge \gamma (\alpha^k_i)^2$ for all $k \in K''$, we get~\eqref{lim_alpha_2}.
\end{proof}

By taking into account Proposition \ref{prop:lim_alpha}, it is easy to get the following corollary, related to the sequences of intermediate points $\{z_i^k\},\ i=1,\dots, \bar m$.

\begin{corollary}\label{corol:lim_z}
Let $\{y^k\}$ be a sequence of points produced by \DFSIMPL. Then,
\[
\lim_{k \to \infty} \|y^k - z^k_i\| = 0, \quad i = 1,\ldots,\bar m.
\]
\end{corollary}

 We now give the proof of another important result for the global convergence analysis. More specifically, we show that starting stepsizes $\hat \alpha^k_i$ considered in the algorithm go to zero as well.

\begin{proposition}\label{prop:lim_hat_alpha}
Let $\{y^k\}$ be a sequence of points produced by \DFSIMPL. Then,
\[
\lim_{k \to \infty} \hat \alpha^k_i = 0, \quad i = 1,\ldots,\bar m.
\]
\end{proposition}

\begin{proof}
For every fixed $i \in \{1,\ldots,\bar m\}$, we partition the iterations into three subsets $K_1$, $K_2$ and $K_3$ such that
\begin{equation}\label{K1K2K3}
\alpha^k_i \ne 0 \Leftrightarrow k \in K_1, \quad \alpha^k_i = 0, i \ne j_k \Leftrightarrow k \in K_2 \quad \text{and} \quad i = j_k \Leftrightarrow k \in K_3.
\end{equation}
From the instructions of the algorithm, we have that
\begin{align}
\hat \alpha^{k+1}_i = \alpha^k_i \ge \hat \alpha^k_i, \quad & \forall\ k \in K_1, \label{property_K1} \\
\hat \alpha^{k+1}_i = \theta \hat \alpha^k_i < \hat \alpha^k_i, \quad & \forall\ k \in K_2, \label{property_K2} \\
\hat \alpha^{k+1}_i = \min\{\hat \alpha^{k+1}_h, \hat \alpha^k_i\} \le \hat \alpha^k_i, \quad h \in \{1,\ldots,\bar m\} \setminus \{j_k\}, \quad & \forall\ k \in K_3. \label{property_K3}
\end{align}

If $K_1$ is an infinite subset, using~\eqref{property_K1} and Proposition~\ref{prop:lim_alpha} we obtain
\begin{equation}\label{lim_hat_alpha_proof}
\lim_{\substack{k \to \infty \\ k \in K_1}} \hat \alpha^{k+1}_i = 0,
\end{equation}
which, combined with~\eqref{property_K2} and~\eqref{property_K3}, yields to the desired result.
Therefore, in the rest of the proof we assume $K_1$ to be a finite set.

First, consider the case where $K_3$ is a finite set, that is, there exists $\bar k$ such that $k \in K_2$ for all $k \ge \bar k$.
For each $k \in K_2$, define $l_k$ as the largest iteration index such that $l_k < k$ and $l_k \in K_1$ (if it does not exist, we let $l_k = 0$).
Also define $q_k$ as the number of iterations belonging to $K_3$ between $l_k$ and $k$.
Therefore, there are $k-l_k-q_k$ iterations belonging to $K_2$ between $l_k$ and $k$.
From~\eqref{property_K2}--\eqref{property_K3}, it follows that
\[
\hat \alpha^{k+1}_i \le \theta^{k-l_k-q_k} \, \hat \alpha^{l_k+1}_i.
\]
Using the fact that both $l_k$ and $q_k$ are bounded from above (since both $K_1$ and $K_3$ are finite sets),
we have that $\lim_{\substack{k \to \infty \\ k \in K_2}} \theta^{k-l_k-q_k} = 0$. Therefore,
$\lim_{\substack{k \to \infty \\ k \in K_2}} \hat \alpha^{k+1}_i = \lim_{k \to \infty} \hat \alpha^{k+1}_i = 0$ and the desired result is obtained.

Now, we consider the case where $K_3$ is an infinite set and we distinguish two subcases.
If $K_2$ is an infinite set, from~\eqref{property_K2} and~\eqref{property_K3} we have that
$\lim_{\substack{k \to \infty \\ k \in K_2 \cup K_3}} \hat \alpha^{k+1}_i = \lim_{k \to \infty} \hat \alpha^{k+1}_i = 0$
and the desired result is obtained.
Else (i.e., if $K_2$ is a finite set), there exists $\tilde k$ such that $k \in K_3$ for all $k \ge \tilde k$ and, picking any index $t \in \{1,\ldots,\bar m\} \setminus \{i\}$,
we can partition the iterations into three subsets $Q_1$, $Q_2$ and $Q_3$ such that
\[
\alpha^k_t \ne 0 \Leftrightarrow k \in Q_1, \quad \alpha^k_t = 0, t \ne j_k \Leftrightarrow k \in Q_2 \quad \text{and} \quad t = j_k \Leftrightarrow k \in Q_3.
\]
Since $i \in K_3$ for all $k \ge \tilde k$, we have that $Q_3$ is a finite set and, with the same arguments given above for the case where $K_3$ is a finite set,
we obtain that $\lim_{k \to \infty} \hat \alpha^k_t = 0$.
Using the fact that, from the instructions of the algorithm,
\[
\hat \alpha^{k+1}_i \le \min_{h \in \{1,\ldots,\bar m\} \setminus \{i\}} \hat \alpha^{k+1}_h, \quad \forall\ k \in K_3,
\]
the desired result is obtained.
\end{proof}

Now, we can state the main convergence result related to \DFSIMPL. In particular, we show that every limit point of the sequence $\{y^k\}$ generated by the proposed method is stationary for Problem~\eqref{prob_sim}.

\begin{theorem}\label{th:conv_basic}
Let $\{y^k\}$ be a sequence of points produced by \DFSIMPL. Then, every limit point $y^*$ is stationary for Problem~\eqref{prob_sim}.
\end{theorem}

\begin{proof}
Let us consider a subsequence such that
\[
\lim_{k \to \infty, \, k \in K} y^k = y^*,
\]
with $K \subseteq \{1,2,\ldots\}$.
Since the set of indices $\{1,\ldots,\bar m\}$ is finite, it is possible to consider a further subsequence,
still denoted by $\{y^k\}_K$ without loss of generality, such that $j_k = \hat \jmath$ for all $k \in K$.

We first show that a real number $\rho > 0$ and an iteration $\bar k \in K$ exist such that
\begin{equation}\label{y_j_pos}
(z^k_i)_{\hat \jmath} \ge \rho, \quad \forall\ k \ge \bar k, \, k \in K, \quad i = 1,\ldots, \bar m.
\end{equation}
Let $\bar h$ be any index such that $y^*_{\bar h} > 0$ and let $\rho$ be a positive real number such that
$y^*_{\bar h} \ge (4/\tau)\rho$.
For all sufficiently large $k \in K$ we have that $y^k_{\bar h} \ge (2/\tau)\rho$ and, recalling how we choose the index $j_k$ (see line~3 of Algorithm~\ref{alg:basic}),
for all sufficiently large $k \in K$ we obtain
\[
y^k_{\hat \jmath} \ge \tau \max_{i=1,\ldots,\bar m} y^k_i \ge \tau y^k_{\bar h} \ge 2 \rho.
\]
Using Corollary~\ref{corol:lim_z}, it follows that
\begin{equation}\label{lim_z}
\lim_{k \to \infty, \, k \in K} z^k_i = y^*, \quad i = 1,\ldots,\bar m,
\end{equation}
implying that~\eqref{y_j_pos} holds and $y^*_{\hat \jmath} > 0$.

From~\eqref{kkt} we have that $y^*$ is a stationary point if and only if a $\lambda^* \in \R$ exists such that
\[
\nabla_i \varphi(y^*)
\begin{cases}
\ge \lambda^*, \quad & \text{if } y^*_i = 0, \\
= \lambda^*,   \quad & \text{if } y^*_i > 0,
\end{cases}
\]
for all $i = 1,\ldots,\bar m$.
Since we have just proved that $y^*_{\hat \jmath} > 0$,
in our case we have that $y^*$ is a stationary point if and only if
\[
\nabla_i \varphi(y^*)
\begin{cases}
\ge \nabla_{\hat \jmath} \varphi(y^*), \quad & \text{if } y^*_i = 0, \\
= \nabla_{\hat \jmath} \varphi(y^*),   \quad & \text{if } y^*_i > 0,
\end{cases}
\]
for all $i = 1,\ldots,\bar m$.

So, assuming by contradiction that $y^*$ is not a stationary point, an index $t$ must exist such that one of the following two cases holds.
\begin{enumerate}[label=(\roman*), leftmargin=*]
\item $y^*_t = 0$ and $\nabla_t \varphi(y^*) < \nabla_{\hat \jmath} \varphi(y^*)$.
    By the mean value theorem, we can write
    \[
    \varphi(z^k_t - \hat \alpha^k_t (e_t - e_{\hat \jmath})) - \varphi(z^k_t) =
    -\hat \alpha^k_t \nabla \varphi(u^k_t)^T (e_t - e_{\hat \jmath}),
    \]
    where $u^k_t = z^k_t - \omega^k_t \hat \alpha^k_t (e_t - e_{\hat \jmath})$ and $\omega^k_t \in (0,1)$.
    Using Proposition~\ref{prop:lim_hat_alpha} and~\eqref{lim_z}, we have that
    \[
    \lim_{k \to \infty, \, k \in K} \nabla \varphi(u^k_t)^T (e_t - e_{\hat \jmath}) = \nabla \varphi(y^*)^T (e_t - e_{\hat \jmath}) = \nabla_t \varphi(y^*) - \nabla_{\hat \jmath} \varphi(y^*) < 0.
    \]
    It follows that, for all sufficiently large $k \in K$,
    \begin{equation}\label{f_incr}
    \varphi(z^k_t - \hat \alpha^k_t (e_t - e_{\hat \jmath})) > \varphi(z^k_t).
    \end{equation}

    Now, using Proposition~\ref{prop:lim_alpha} we have that, for all sufficiently large $k \in K$,
    \begin{equation}\label{z_feas}
    z^k_t + \hat \alpha^k_t (e_t - e_{\hat \jmath}) \in \Delta_{\bar m-1}.
    \end{equation}
    Taking into account~\eqref{z_feas} and the instructions of the algorithm, for all sufficiently large $k \in K$ either
    $\alpha^k_t = 0$ and
    $
    \varphi(z^k_t + \hat \alpha^k_t (e_t - e_{\hat \jmath})) > \varphi(z^k_t) - \gamma (\hat \alpha^k_t)^2,
    $
    or $\alpha^k_t \ne 0$.
    In the latter case, combining~\eqref{f_incr} and~\eqref{z_feas} we have that, for all sufficiently large $k \in K$, the algorithm does not move along the direction $e_{\hat \jmath}-e_t$, and then,
    $d^k_t = e_t-e_{\hat \jmath}$.
    Using Proposition~\ref{prop:lim_hat_alpha} we also get that, for all sufficiently large $k \in K$,
    $
    z^k_t + \frac{\alpha^k_t}{\delta} (e_t - e_{\hat \jmath}) \in \Delta_{\bar m-1}.
    $
    Therefore, taking into account the \texttt{Line Search Procedure} we have that
    \[
    \varphi\Bigl(z^k_t + \frac{\alpha^k_t}{\delta} (e_t - e_{\hat \jmath})\Bigr) > \varphi(z^k_t) - \gamma \Bigl(\frac{\alpha^k_t}{\delta}\Bigr)^2,
    \]
    for all sufficiently large $k \in K$.
    Using the mean value theorem in the two above inequalities, we have that either
    \[
    \nabla \varphi(\nu^k_t)^T (e_t - e_{\hat \jmath}) > - \gamma \hat \alpha^k_t \quad \text{or} \quad \nabla \varphi(s^k_t)^T (e_t - e_{\hat \jmath}) > - \gamma \frac{\alpha^k_t}{\delta},
    \]
    where $\nu^k_t = z^k_t + \pi^k_t \hat \alpha^k_t (e_t - e_{\hat \jmath})$, with $\pi^k_t \in (0,1)$ and $s^k_t = z^k_t + \eta^k_t [\alpha^k_t/\delta] (e_t - e_{\hat \jmath})$, with $\eta^k_t \in (0,1)$.
    Using Proposition~\ref{prop:lim_alpha}, Proposition~\ref{prop:lim_hat_alpha} and the continuity of $\nabla \varphi$, we can take the limits for $k \to \infty$, $k \in K$,
    and we obtain $\nabla \varphi(y^*)^T (e_t - e_{\hat \jmath}) \ge 0$, contradicting the fact that $\nabla_t \varphi(y^*) < \nabla_{\hat \jmath} \varphi(y^*)$.
\item $y^*_t > 0$ and $\nabla_t \varphi(y^*) \ne \nabla_{\hat \jmath} \varphi(y^*)$. First note that, since $y^*_{\hat \jmath} > 0$, necessarily $y^*_t < 1$
    and, consequently, for all sufficiently large $k \in K$ both the directions $\pm (e_t - e_{\hat \jmath})$ are feasible at $z^k_t$.

    Now, assume that $\nabla_t \varphi(y^*) < \nabla_{\hat \jmath} \varphi(y^*)$. Reasoning as in case~(i), we obtain $\nabla \varphi(y^*)^T (e_t - e_{\hat \jmath}) \ge 0$,
    thus getting a contradiction. Then, necessarily $\nabla_t \varphi(y^*) > \nabla_{\hat \jmath} \varphi(y^*)$ but, repeating again the same reasoning as in case~(i)
    with minor modifications, we obtain $\nabla \varphi(y^*)^T (e_t - e_{\hat \jmath}) \le 0$, getting a new contradiction and thus proving the desired result.
\end{enumerate}
\end{proof}

\subsection{Choice of the stopping condition}\label{subsec:stopping_cond_alg_basic}
Now, we describe the stopping condition employed in \DFSIMPL.
As we will see in the next section, this is a
key tool for the theoretical analysis of the  general inner approximation scheme that embeds
\DFSIMPL\ as solver of the reduced problem.
Moreover, under the assumption that $\nabla f$ is Lipschitz continuous, we will show that
the stationarity error of the solution returned by \DFSIMPL\ is upper bounded by a term that depends on the tolerance chosen in the stopping criterion (see Theorem~\ref{th:error_stat} below).

Given a tolerance $\epsilon > 0$, a standard choice in direct search methods is to terminate the algorithm when a suitable steplength control parameter falls below $\epsilon$.
In our case, this means that $\hat \alpha^k_i \le \epsilon$, $i = 1,\ldots,\bar m$.
Additionally, we prevent each $\hat \alpha^k_i$ to become smaller than $\epsilon$. In particular, at line $9$ of Algorithm~\ref{alg:basic}
instead of setting $\hat \alpha^{k+1}_i = \theta \hat \alpha^k_i$ we use the following rule:
\begin{equation}\label{upd_rule_hat_alpha}
\hat \alpha^{k+1}_i = \max\{\theta \hat \alpha^k_i, \epsilon\}.
\end{equation}
We see that, if $\epsilon = 0$, we have exactly the rule reported in the scheme of Algorithm~\ref{alg:basic}.
In order to stop the algorithm, we also require that no progress is made along any feasible direction, that is $\alpha^k_i = 0$ for all $i \ne j_k$.

Summarizing, given $\epsilon>0$, we use~\eqref{upd_rule_hat_alpha} to update each $\hat \alpha^{k+1}_i$ at line $9$ of Algorithm~\ref{alg:basic}
and we terminate the algorithm at the first iteration $k$ such that
\begin{equation}\label{stopping_cond_alg_basic}
\hat \alpha^k_i = \epsilon, \; \forall\ i \in \{1,\ldots,\bar m\}, \qquad \text{and} \qquad \alpha^k_i = 0, \; \forall\ i \ne j_k.
\end{equation}
In the next proposition it is shown that this stopping condition is well defined.

\begin{proposition}\label{prop:stopping_cond_alg_basic}
Given $\epsilon > 0$, the stopping condition~\eqref{stopping_cond_alg_basic} is satisfied
by \DFSIMPL\ after a finite number of iterations.
\end{proposition}

\begin{proof}
First note that, in view of~\eqref{upd_rule_hat_alpha}, we have that
\[
\hat \alpha^k_i \ge \epsilon, \quad \forall\ k \ge 0, \quad \forall\ i \in \{1,\ldots,\bar m\}.
\]

Now we show that an iteration $\bar k$ exists such that
\begin{equation}\label{hat_alpha_eq_eps}
\hat \alpha^k_i = \epsilon, \quad \forall\ k \ge \bar k, \quad \forall\ i \in \{1,\ldots,\bar m\}.
\end{equation}
Proceeding by contradiction, assume that this is not true.
Then, an infinite subsequence $\{y^k\}_{K \subseteq \{0,1,\ldots\}}$ and an index $i \in \{1,\ldots,\bar m\}$ exist such that
\begin{equation}\label{contr_prood_stopping_cond}
\hat \alpha^k_i > \epsilon, \; \forall\ k \in K.
\end{equation}
Using the same arguments given in the proof of Proposition~\ref{prop:lim_alpha}, we have that
\begin{equation}\label{lim_alpha_proof_stopping_cond}
\lim_{k \to \infty} \alpha^k_i = 0.
\end{equation}
Then, to obtain the desired contradiction with~\eqref{contr_prood_stopping_cond} we can reason similarly as in the proof of Proposition~\ref{prop:lim_hat_alpha}, with minor changes that are now described.
Define $K_1$, $K_2$ and $K_3$ as in~\eqref{K1K2K3}. The following relations hold:
\begin{align}
\hat \alpha^{k+1}_i = \alpha^k_i \ge \hat \alpha^k_i \ge \epsilon, \quad & \forall\ k \in K_1, \label{property_K1_stopping_cond} \\
\epsilon \le \hat \alpha^{k+1}_i = \max\{\theta \hat \alpha^k_i, \epsilon\} \le \hat \alpha^k_i, \quad & \forall\ k \in K_2, \label{property_K2_stopping_cond} \\
\epsilon \le \hat \alpha^{k+1}_i \le \hat \alpha^k_i, \quad & \forall\ k \in K_3. \label{property_K3_stopping_cond}
\end{align}
From~\eqref{property_K1_stopping_cond} and~\eqref{lim_alpha_proof_stopping_cond}, we see that $K_1$ cannot be an infinite set.
So, we only have to consider the cases where $K_1$ is finite.
If $K_3$ is also a finite set (and then $K_2$ is an infinite set), we can define $l_k$ ad $q_k$ as in the proof of Proposition~\ref{prop:lim_hat_alpha} and for all $k \in K_2$ we obtain
$
\epsilon \le \hat \alpha^{k+1}_i \le \max\{\theta^{k-l_k-q_k} \, \hat \alpha^{l_k+1}_i,\epsilon\}.$
It follows that $\hat \alpha^k_i = \epsilon$ for all sufficiently large iterations.
If $K_3$ is a infinite set, we distinguish two subcases.
If $K_2$ is also an infinite set, from~\eqref{property_K2_stopping_cond} and~\eqref{property_K3_stopping_cond}
again we have $\hat \alpha^k_i = \epsilon$ for all sufficiently large iterations.
Else (i.e., if $K_2$ is a finite set), we can reason as in the last part of the proof of Proposition~\ref{prop:lim_hat_alpha},
defining in the same way the index $t$ and the three subsets $Q_1$, $Q_2$ and $Q_3$,
obtaining that $Q_3$ is a finite set and, with the same arguments given above for the case where $K_3$ is a finite set,
$\alpha^k_t = \epsilon$ for all sufficiently large iterations.
Using the fact that $\epsilon \le \hat \alpha^{k+1}_i \le \min_{h \in \{1,\ldots,\bar m\} \setminus \{i\}} \hat \alpha^{k+1}_h$ for all $k \in K_3$,
also in this case we obtain that $\hat \alpha^k_i = \epsilon$ for all sufficiently large iterations.
So,~\eqref{hat_alpha_eq_eps} holds.

Finally, to conclude the proof now we show that, for all sufficiently large iterations, $\alpha^k_i = 0$ for all $i \ne j_k$.
Proceeding by contradiction, assume that this is not true.
Then, an infinite subsequence $\{y^k\}_{K \subseteq \{0,1,\ldots\}}$ and an index $i \in \{1,\ldots,\bar m\}$ exist such that $\alpha^k_i > 0, \; \forall\ k \in K.$ From the instructions of the algorithm we have that
$
\hat \alpha^{k+1}_i = \alpha^k_i \ge \hat \alpha^k_i \ge \epsilon, \quad \forall\ k \in K.
$
Since, using again the same arguments given in the proof of Proposition~\ref{prop:lim_alpha}, we have that $\lim_{k \to \infty} \alpha^k_i = 0$,
we thus obtain a contradiction.
\end{proof}

\subsection{Additional stationarity results}\label{subsec:add_stat_res}
Using the stopping condition~\eqref{stopping_cond_alg_basic} with a given tolerance $\epsilon>0$,
we want to show that, when $\nabla f$ is Lipschitz continuous, the solution $\bar y$ returned by \DFSIMPL\ satisfies the following condition:
\begin{equation}\label{stat:ub}
\max_{y \in \Delta_{\bar m-1}} -\nabla \varphi(\bar y)^T (y - \bar y) \le C \epsilon,
\end{equation}
for a suitable constant $C>0$.
Note that $\bar y$ is stationary if and only if $$\max_{y \in \Delta_{\bar m-1}} -\nabla \varphi(\bar y)^T (y - \bar y)=0,$$ thus the quantity given in \eqref{stat:ub} provides a measure for the stationarity error at $\bar y$.

The desired error bound can be obtained by suitably adapting standard results of direct-search methods for linearly constrained problems (see~\cite{kolda2007stationarity,lewis2000pattern}). In order to carry out the analysis,  we first need to recall a few definitions and to point out some geometric properties of the search directions used in \DFSIMPL.

To this extent, it is convenient to consider a reformulation of Problem~\eqref{prob_sim} as an inequality constrained problem of the following form:

\begin{equation}\label{prob_sim_2}
\begin{split}
& \displaystyle\min_y \varphi(y) \\
& s.t. \quad c_i^T y \le b_i, \quad i = 1,\ldots,\bar m+2,
\end{split}
\end{equation}
where $c_1 = e$, $c_2 = -e$, $c_{i+2} = -e_i$, $i=1,\ldots,\bar m$,
and $b_1 = 1$, $b_2 = -1$, $b_{i+2} = 0$, $i=1,\ldots,\bar m$.

Let us recall the definition of \textit{active constraints}, \textit{tangent cone} and \textit{normal cone} for the above problem.

\begin{definition}
Let $y$ be a feasible point of Problem~\eqref{prob_sim_2}.
We say that a constraint $c_i$ is \textit{active} at $y$ if $c_i^T y = b_i$.
We also indicate with $Z(y)$ the index set of active constraints at $y$, that is,
$
Z(y) = \{i \colon c_i^T y = b_i\}.
$
\end{definition}

\begin{definition}
Let $y$ be a feasible point of Problem~\eqref{prob_sim_2}.
We indicate with $N(y)$ the normal cone at $y$, defined as the cone generated by the active constraints at~$y$:
\[
N(y) = \{v \in \R^{\bar m} \colon v = \sum_{i \in Z(y)} \lambda_i c_i, \, \lambda_i \ge 0, \, i \in Z(y)\}.
\]

We also indicate with $T(y)$ the tangent cone at $y$, defined as the polar of $N(y)$:
\[
T(y) = \{v \in \R^{\bar m} \colon v^T d \le 0, \, \forall\ d \in N(y)\}.
\]
\end{definition}

It is easy to see that the tangent cone $T(y)$ at
a feasible point $y$ of Problem~\eqref{prob_sim_2} can be equivalently described as follows:
\begin{equation}\label{lemma:tan_cone}
T(y) = \{v \in \R^{\bar m} \colon e^T v = 0, \, v_i \ge 0, \, i \colon y_i=0\}.
\end{equation}

Now, for every iteration $k$ of \DFSIMPL, let $D^k$ be the set of all the search directions in $\{\pm (e_i-e_{j_k}), \, i = 1,\ldots,\bar m, \, i \ne j_k\}$
that are feasible at $y^k$ (where a search direction $d$ is said to be \textit{feasible} at $y^k$ if there exists $\bar \alpha > 0$ such that
$y + \alpha d \in \Delta_{\bar m-1}$ for all $\alpha \in (0,\bar \alpha]$).
The next remark describes an important property of the set $D^k$.
%lemma shows that $D^k$ is a set of generators of $T(y^k)$.
\begin{remark}\label{lemma:gen_tan_come}
For every iteration $k$ of \DFSIMPL, $D^k$ is a set of generators for the tangent cone $T(y^k)$.
\end{remark}

From now on, given a vector $v$ and a convex cone $\cal C$, we define $v_{\cal C}$ as the projection of $v$ onto $\cal C$.
Thus, $v_{T(y)}$ is the projection of $v$ onto $T(y)$ and $v_{N(y)}$ is the projection of $v$ onto $N(y)$.

Before stating the desired result, we also need the following lemma to show that, for any vector $v \in \R^{\bar m}$ and for any iteration $k$, a direction $d \in D^k$ exists such that
the inner product $v^T d$ is lower bounded by $\norm{v_{T(y^k)}}$ up to some constant.

\begin{lemma}\label{lemma:cone_measure}
For every iteration $k$ of \DFSIMPL, we have that
\[
\displaystyle\max_{d \in D^k} v^T d  \ge \frac{\norm{v_{T(y^k)}}}{2(\bar m-1)}, \quad \forall\ v \in \R^{\bar m}.
\]
\end{lemma}

\begin{proof}
We first observe that any vector $\sigma \in T(y^k)$ can be expressed as a non-negative linear combination of the vectors in $D^k$ with coefficients
$|\sigma_i| \le \norm {\sigma}$, that is
\begin{equation}\label{w_tan_cone}
     \sigma =  \sum_{\substack{i \ne j_k \\ i \colon \sigma_i \ne 0}}  \text{sign}(\sigma_i)(e_i - e_{j_k})|\sigma_i|.
\end{equation}
Now, pick any vector $v \in \R^{\bar m}$ and, for the sake of simplicity, define $u_1,\ldots,u_{|D^k|}$ the directions in $D^k$.
It follows that there exist non-negative coefficients $\lambda_1,\ldots,\lambda_{|D^k|}$, with $0 \le \lambda_i \le \norm{v_{T(y^k)}}$, $i = 1,\ldots,|D^k|$, such that
$
v_{T(y^k)} = \sum_{i=1}^{|D^k|} \lambda_i u_i,
$
and then, $\displaystyle{v^T v_{T(y^k)} = \sum_{i=1}^{|D^k|} \lambda_i v^T u_i}$.
Therefore, an index $i \in \{1,\ldots,|D^k|\}$ exists such that
\[
\lambda_i v^T u_i \ge \frac1{|D^k|} v^T v_{T(y^k)} \ge \frac1{2(\bar m-1)} v^T v_{T(y^k)} \ge \frac1{2(\bar m-1)} \|v_{T(y^k)}\|^2,
\]
where the last inequality follows from the property of the projection.
Since we have $0 \le \lambda_i~\le~\norm{v_{T(y^k)}}$, the result is obtained.
\end{proof}

We are finally ready to provide a bound on the stationarity error for the solution returned by \DFSIMPL.
\begin{theorem}\label{th:error_stat}
Assume that $\nabla \varphi$ is Lipschitz continuous with constant $L$ and
the stopping condition~\eqref{stopping_cond_alg_basic} is used with a given tolerance $\epsilon > 0$.
Then, the solution $\bar y$ returned by \DFSIMPL\ is such that
\[
\max_{y \in \Delta_{\bar m-1}} -\nabla \varphi(\bar y)^T (y - \bar y) \le C \epsilon,
\]
where $C = 2 \sqrt2 (\bar m-1) (2L + \gamma)$.
\end{theorem}

\begin{proof}
Let $k$ be the last iteration of \DFSIMPL, so that $y^k = \bar y$.
In view of Lemma~\ref{lemma:cone_measure}, used with $v = -\nabla \varphi(\bar y)$, we have that a $d \in D^k$ exists such that
\begin{equation}\label{ineq_stat_proof}
-\nabla \varphi(\bar y)^T d \ge \frac{\bigl\|[-\nabla \varphi(\bar y)]_{T(\bar y)}\bigr\|}{2(\bar m-1)}.
\end{equation}
Since in~\eqref{stopping_cond_alg_basic} we require $\alpha^k_i = 0$ for all $i \ne j_k$
(i.e., no progress is made along any feasible direction),
from the instructions of the algorithm and the \texttt{Line Search Procedure} we have that
\begin{equation}\label{unsuccessful_k}
\varphi(\bar y + \alpha d) > \varphi(\bar y) - \gamma \alpha^2,
\end{equation}
with
\begin{equation}\label{alpha_proof}
0 < \alpha \le \epsilon,
\end{equation}
where the last inequalities for $\alpha$ follow from the fact that each $\hat \alpha^k_i$ is required to be equal to $\epsilon$ in~\eqref{stopping_cond_alg_basic}.
By the mean value theorem, $\varphi(\bar y + \alpha d) - \varphi(\bar y) = \alpha \nabla \varphi(\bar y + \eta \alpha d)^T d$, for some $\eta \in (0,1)$.
Thus, from~\eqref{unsuccessful_k}, we obtain $\alpha \nabla \varphi(\bar y + \eta \alpha d)^T d + \gamma \alpha^2 > 0$.
Dividing both terms by $\alpha$, we get
$
\nabla \varphi(\bar y + \eta \alpha d)^T d + \gamma \alpha > 0.
$
Now, we subtract $\nabla \varphi(\bar y)^T d$ to both terms of the above inequality, obtaining
\[
[\nabla \varphi(\bar y + \eta \alpha d) - \nabla \varphi(\bar y)]^T d + \gamma \alpha > -\nabla \varphi(\bar y)^T d.
\]
Using the fact that $\nabla \varphi$ is Lipschitz continuous, we have
$[\nabla \varphi(\bar y + \eta \alpha d) - \nabla \varphi(\bar y)]^T d \le L \eta \alpha \norm d^2 \le 2L \alpha$,
where the last inequality follows from the fact that $\eta \in (0,1)$ and $\norm d = \sqrt 2$. Then,
$
2L \alpha + \gamma \alpha > -\nabla \varphi(\bar y)^T d.
$
Combining this inequality with~\eqref{ineq_stat_proof} and~\eqref{alpha_proof}, we get
$
\bigl\|[-\nabla \varphi(\bar y)]_{T(\bar y)}\bigr\| < 2\epsilon(\bar m-1) (2L + \gamma).
$
To conclude the proof, we thus have to show that
\begin{equation}\label{upper_bound_g}
\max_{y \in \Delta_{\bar m-1}} -\nabla \varphi(\bar y)^T (y - \bar y) \le \sqrt 2 \bigl\|[-\nabla \varphi(\bar y)]_{T(\bar y)}\bigl\|.
\end{equation}
Since, by polar decomposition, every vector $v \in \R^{\bar m}$ can be written as $v = v_{T(\bar y)} + v_{N(\bar y)}$ (see, e.g.,~\cite{zarantonello1971projections})
we have $-\nabla \varphi(\bar y) = [-\nabla \varphi(\bar y)]_{T(\bar y)} + [-\nabla \varphi(\bar y)]_{N(\bar y)}$.
Therefore, for any $y \in \Delta_{\bar m-1}$ we can write
\begin{equation}\label{upper_bound_g_1}
-\nabla \varphi(\bar y)^T (y - \bar y) = [-\nabla \varphi(\bar y)]_{T(\bar y)}^T (y-\bar y) + [-\nabla \varphi(\bar y)]_{N(\bar y)}^T (y-\bar y).
\end{equation}
In order to upper bound the right-hand side term of the above inequality, we first write
\begin{equation}\label{upper_bound_g_2}
[-\nabla \varphi(\bar y)]_{T(\bar y)}^T (y-\bar y) \le \bigl\|[-\nabla \varphi(\bar y)]_{T(\bar y)}\bigr\| \norm{y-\bar y} \le \sqrt 2 \bigl\|[-\nabla \varphi(\bar y)]_{T(\bar y)}\bigl\|,
\end{equation}
where the last inequality follows from the fact that both $y$ and $\bar y$ belong to $\Delta_{\bar m-1}$.
Moreover, we have that $y-\bar y \in T(\bar y)$.
%This follows from definition~\eqref{lemma:tan_cone}, using again the fact that both $y$ and $\bar y$ belong to $\Delta_{n-1}$,
%which implies that $e^T(y-\bar y) = e^T y - e^T \bar y~=~0$ and $y_i - \bar y_i \ge 0$ if $\bar y_i = 0$.
Therefore, from the definition of the tangent cone, we also have that
\begin{equation}\label{upper_bound_g_3}
[-\nabla \varphi(\bar y)]_{N(\bar y)}^T (y-\bar y) \le 0.
\end{equation}
From~\eqref{upper_bound_g_1}, \eqref{upper_bound_g_2} and~\eqref{upper_bound_g_3}, we conclude that
\[
-\nabla \varphi(\bar y)^T (y - \bar y) \le \sqrt 2 \bigl\|[-\nabla \varphi(\bar y)]_{T(\bar y)}\bigl\|, \quad \forall\ y \in \Delta_{\bar m-1},
\]
that is~\eqref{upper_bound_g} holds and the result is obtained.
\end{proof}

\section{Optimize, Refine \& Drop (ORD) Algorithm}
In principle, we might use barycentric coordinates to represent the feasible set of Problem~\eqref{prob}, thus obtaining a new problem of the form given in~\eqref{prob_sim0} that might be solved by \DFSIMPL\ (or any other solver for linearly constrained optimization). Unfortunately,
since the number of variables in Problem~\eqref{prob_sim0} is the same as the number of atoms in $\mathcal A$, when $|\mathcal A|$ increases
(keep in mind that this is often the case in our context), it gets hard to obtain a reasonable solution within the given budget of function evaluations. We further notice that
in our context good points usually lie in small dimensional faces of the feasible set (i.e., only a small number of atoms is needed to assemble those points). This is the reason why we propose an inner approximation scheme to tackle the problem.

\begin{algorithm}[H]
\caption{\texttt{Optimize, Refine \& Drop (ORD) Algorithm}}
\label{alg:general}
\begin{algorithmic}
{\footnotesize
\par\vspace*{0.1cm}
\item[]\hspace*{-0.1truecm}$\,\,\,1$\hspace*{0.1truecm} Choose $\{\epsilon^k\} \searrow 0$, $\mathcal A^0 \subseteq \mathcal A$, $a_{i_0} \in \mathcal A^0$,
                                  set $x^0 = a_{i_0}$, $y^0 = e_{i_0} \in \R^{|\mathcal A^0|}$, $\hat \mu^0 \in (0,1)$, \\
                                  \hspace*{-0.1truecm}$\,\,\,\,\,\,$\hspace*{0.1truecm} $\gamma > 0$ and $\theta \in (0,1)$
\item[]\hspace*{-0.1truecm}$\,\,\,2$\hspace*{0.1truecm} For $k=0,1,\ldots$
\vspace*{0.2truecm}
\item[]$\,\,\, $\hspace*{0.65truecm} \textbf{Optimize Phase}
\item[]\hspace*{-0.1truecm}$\,\,\,3$\hspace*{0.9truecm} Let $A^k$ be the matrix with the atoms in $\mathcal A^k$ as columns (so that $x^k = A^k y^k$)
\item[]\hspace*{-0.1truecm}$\,\,\,4$\hspace*{0.9truecm} Run \DFSIMPL\ from $y^k$ to compute an approximate solution $\bar y^k$ of \\
                                   \hspace*{-0.1truecm}$\,\,\,\,\,\,$\hspace*{0.9truecm} Problem~\eqref{eqprobconv} with tolerance $\epsilon^k$
\item[]\hspace*{-0.1truecm}$\,\,\,5$\hspace*{0.9truecm} Set $\bar x^k=A^k \bar y^k$
\vspace*{0.2truecm}
\item[]$\,\,\, $\hspace*{0.65truecm} \textbf{Refine Phase}
\item[]\hspace*{-0.1truecm}$\,\,\,6$\hspace*{0.9truecm} If there exists an index $i_k \in \{1,\ldots,m\}$ and a scalar $\mu^k \in [\hat \mu^k,1]$ such that
                                  \[
                                  f(\bar x^k + \mu^k (a_{i_k} - \bar x^k)) \le f(\bar x^k) - \gamma (\mu^k)^2, \quad a_{i_k} \in {\mathcal A}\setminus {\mathcal A}^k,
                                  \]
                                  \hspace*{-0.1truecm}$\,\,\,\,\,\,$\hspace*{0.9truecm} then set $x^{k+1} = \bar x^k + \mu^k (a_{i_k} - \bar x^k)$,
                                  $\mathcal R^k = \{a_{i_k}\}$ and $\hat \mu^{k+1} = \hat \mu^k$
\item[]\hspace*{-0.1truecm}$\,\,\,7$\hspace*{0.9truecm} Else set $x^{k+1} = \bar x^k$, $\mathcal R^k = \emptyset$ and $\hat \mu^{k+1} = \theta \hat \mu^k$
\vspace*{0.2truecm}
\hspace*{0.8truecm}
\item[]$\,\,\, $\hspace*{0.65truecm} \textbf{Drop Phase}
\item[]\hspace*{-0.1truecm}$\,\,\,8$\hspace*{0.9truecm} Choose a subset $\mathcal D^k \subseteq \{a\in \mathcal{A}^k \text{ such that } a=A^k e_h \text{ and } \bar y^k_h=0\}$
\vspace*{0.25truecm}
\item[]\hspace*{-0.1truecm}$\,\,\,9$\hspace*{0.9truecm} Let $\mathcal A^{k+1} = \mathcal A^k \cup \mathcal R^k \setminus \mathcal D^k$,
                             and set $y^{k+1} \in \Delta_{|\mathcal A^{k+1}|-1}$ such that\\
                             \hspace*{-0.1truecm}$\,\,\,\,\,\,$\hspace*{0.9truecm} $\displaystyle{x^{k+1} = \sum_{a_i \in \mathcal A^{k+1}} a_i y^{k+1}_i}$
\item[]\hspace*{-0.1truecm}$10$\hspace*{0.1truecm} End for
\par\vspace*{0.1cm}
}
\end{algorithmic}
\end{algorithm}

At a given iteration $k$, our method considers a reduced problem by approximating the set $\mathcal M$ with the convex hull of a set ${\mathcal A}^k \subseteq \mathcal A$, and tries to suitably improve this description by including/removing atoms according to some given rule. We can now describe in depth the three main phases that characterize  our approach.

Let $A^k$ be the matrix whose columns are the atoms in $\mathcal A^k$.
First, in the \textit{Optimize Phase}, we use \DFSIMPL\ to compute an approximate solution of the following reduced problem:
\begin{equation}\label{eqprobconv}
\min_{y\in \Delta_{|\mathcal A^k|-1}} \varphi^k(y),
\end{equation}
where $\varphi^k(y) = f(A^k y)$. In particular, we run \DFSIMPL\ on Problem~\eqref{eqprobconv} until a given tolerance $\epsilon^k$ is reached, according to the stopping condition discussed in Subsection~\ref{subsec:stopping_cond_alg_basic}.
% Namely, Algorithm~\ref{alg:basic} is stopped at the end at the first iteration $l \in \mathcal U$ such that
% \begin{equation}\label{stop_alg_basic}
% \hat \alpha^l_{\text{max}} \le \epsilon^k,
% \end{equation}
% where $\mathcal U$ is the subset of unsuccessful iterations for Algorithm~\ref{alg:basic}, as defined in Subsection~\ref{subsec:add_stat_res},
% and $\hat \alpha^l_{\text{max}}$ is defined as in the statement of Theorem~\ref{th:error_stat}.

In the second phase, the so-called \textit{Refine Phase}, we try to get a better inner description of $\mathcal{M}$ by choosing an atom $a_{i_k} \in \mathcal A \setminus \mathcal A^k$, with $i_k \in \{1,\ldots,m\}$, that guarantees improvement of the objective value (we use  $\mathcal R^k$ to indicate the set that, if non-empty, is a singleton composed by the atom to be added to $\mathcal A^k$).
In practice, we randomly pick the atoms in $\mathcal A \setminus \mathcal A^k$, with no repetition, and we stop when we find one satisfying a sufficient decrease condition.

Finally, in the last phase (\textit{Drop Phase}),
we  get rid of some atoms in ${\mathcal A}^k$ thanks to a simple selection rule
(we will use the notation $\mathcal D^k$ to indicate the set of atoms to be removed from $\mathcal A^k$).
This tool enables us to keep the dimension of the reduced problem small enough along the iterations.

The detailed scheme is reported in Algorithm~\ref{alg:general}. We would like to notice that the parameters $\gamma$ and $\theta$ can be different from those used in \DFSIMPL.

We first introduce suitable optimality conditions for \eqref{prob} that will be exploited in the theoretical analysis of our algorithmic framework.

\begin{proposition}\label{prop:stat_general}
A feasible point $x^*$ of Problem~\eqref{prob} is stationary if and only~if
\[
\nabla f(x^*)^T (a - x^*) \ge 0, \quad \forall\ a \in \mathcal A.
\]
\end{proposition}

Now, we prove that the stepsize used to define the sufficient decrease in the atom selection of the second phase (see line 6 of Algorithm \ref{alg:general}) goes to zero. This result will be needed in the global convergence analysis of the method.

\begin{proposition}\label{prop:lim_hat_mu}
Let $\{x^k\}$ be a sequence of points produced by Algorithm~\ref{alg:general}.
Then,
\[
\lim_{k \to \infty} \hat \mu^k = 0.
\]
\end{proposition}

\begin{proof}
We partition the iterations into two subsets $K_1$ and $K_2$ such that
\begin{equation}\label{K1K2}
\hat \mu^{k+1} = \hat \mu^k \Leftrightarrow k \in K_1\quad \text{and} \quad  \hat \mu^{k+1} = \theta \hat \mu^k \Leftrightarrow k \in K_2,
\end{equation}
that is, the iterations in $K_1$ are those where the test at line~6 of Algorithm~\ref{alg:general} is satisfied,
while the iterations in $K_2$ are those where that test is not satisfied.
From line~6 of Algorithm~\ref{alg:general}, for all $k \in K_1$ we have that
\begin{equation*}
f(x^{k+1}) = f(\bar x^k + \mu^k (a_{i_k} - \bar x^k)) \le f(\bar x^k) - \gamma (\mu^k)^2 \le f(x^k) - \gamma (\mu^k)^2,
\end{equation*}
where $f(\bar x^k) \le f(x^k)$ in the last inequality follows from the fact that $\bar x^k = A^k \bar y^k$ and $\bar y^k$ is obtained from \DFSIMPL\ with a starting point $y^k$ satisfying $x^k = A^k y^k$.
Therefore, if $K_1$ is infinite, using the fact that $f$ is continuous and the feasible set is bounded
it follows that $\{f(x^k)\}$ converges and
\begin{equation}\label{lim_mu_k_subseq}
\lim_{\substack{k \to \infty \\ k \in K_1}} \mu^k = 0.
\end{equation}
Since $\hat \mu^k \le \mu^k$ for all $k \in K_1$, it follows that
$\{\hat \mu^k\}_{K_1} \to 0$.
Taking into account that $\hat \mu^{k+1} = \theta \hat \mu^k$ for all $k \in K_2$, we obtain that the desired holds if $K_1$ is infinite.

If $K_1$ is finite, there exists $\bar k$ such that $k \in K_2$ for all $k \ge \bar k$.
For each $k \in K_2$, define $l_k$ as the largest iteration index such that $l_k < k$ and $l_k \in K_1$ (if it does not exist, we let $l_k = 0$).
Therefore, there are $k-l_k$ iterations belonging to $K_2$ between $l_k$ and $k$, implying that
$
\hat \mu^{k+1} \le \theta^{k-l_k} \, \hat \mu^{l_k+1}.
$
Using the fact that $l_k$ is bounded from above (since $K_1$ is finite),
we have that $\lim_{\substack{k \to \infty \\ k \in K_2}} \theta^{k-l_k} = 0$. Therefore,
$\lim_{\substack{k \to \infty \\ k \in K_2}} \hat \mu^{k+1} = 0$ and the desired result is obtained.
\end{proof}

We thus get the following useful corollary.

\begin{corollary}\label{corol:lim_x}
Let $\{x^k\}$ be a sequence of points produced by Algorithm~\ref{alg:general}. Then,
\[
\lim_{k \to \infty} \|x^{k+1} - \bar x^k\| = 0.
\]
\end{corollary}

\begin{proof}
As in the proof of Proposition~\ref{prop:lim_hat_mu},
let us define $K_1$ and $K_2$ satisfying~\eqref{K1K2}.
If $K_1$ is a finite set, from the instructions of the algorithm we have that an iteration $\tilde k$ exists such that $x^{k+1} = \bar x^k$ for all $k \ge \tilde k$ and the desired result is obtained.
If $K_1$ is an infinite set, by the same arguments used in the proof of Proposition~\ref{prop:lim_hat_mu} we get~\eqref{lim_mu_k_subseq}, that is,
\[
\lim_{\substack{k \to \infty \\ k \in K_1}} \|x^{k+1} - \bar x^k\| = 0,
\]
and the desired result is obtained since $x^{k+1} = \bar x^k$ for all $k \in K_2$.
\end{proof}

In the next theorem,  we prove global convergence of the proposed algorithm.

\begin{theorem}
Let $\{x^k\}$ be a sequence of points produced by Algorithm~\ref{alg:general}.
Then, (at least) one limit point $x^*$ exists such that $x^*$ is stationary for Problem~\eqref{prob}.
\end{theorem}

\begin{proof}
Using Proposition~\ref{prop:lim_hat_mu}, the fact that the feasible set of every reduced Problem~\eqref{eqprobconv} is bounded and the fact that $\mathcal A$ is a finite set,
there exists an infinite subset of iterations $K \subseteq \mathcal \{0,1,\ldots\}$ such that
$$
 \mathcal A^k = \bar {\mathcal A}, \quad \forall\ k \in K; \quad
 \lim_{\substack{k \to \infty \\ k \in K}} \bar y^k = y^*;  \quad
 \mu^{k+1} < \mu^k, \ \forall\ k \in K.
$$
Since $\mathcal A^k$ is constant for all $k \in K$, also the matrix $A^k$ and the function $\varphi^k$ are the same for all $k \in K$,
and let us denote them by $\bar A$ and $\bar \varphi$, respectively.
Hence, we also have
\[
\lim_{\substack{k \to \infty \\ k \in K}} x^k = \bar A y^* = x^*.
\]
Taking into account Proposition~\ref{prop:stat_general}, to obtain the desired result we have to show that
\begin{subequations}
\begin{align}
\nabla f(x^*)^T (a - x^*) \ge 0, \quad & \forall\ a \in \bar {\mathcal A}, \label{stat_proof_1} \\
\nabla f(x^*)^T (a - x^*) \ge 0, \quad & \forall\ a \in \mathcal A \setminus \bar {\mathcal A}. \label{stat_proof_2}
\end{align}
\end{subequations}

To prove~\eqref{stat_proof_1}, for all iterations $k \in K$ consider the points $\bar y^k$,
which are returned by \DFSIMPL\ when the stopping condition~\eqref{stopping_cond_alg_basic} is satisfied.
Since the set of directions used in \DFSIMPL\ is finite, without loss of generality we can assume that, for all $k \in K$,
the set of feasible directions at $\bar y^k$ used in the last iteration of \DFSIMPL\ is the same for all $k \in K$. Let us denote this set of directions by $D$.
Since the stopping condition~\eqref{stopping_cond_alg_basic} requires that
no progress is made along any direction,
from the instructions of \DFSIMPL\ we have that, at any iteration $k \in K$,
\[
\bar \varphi(\bar y^k + \alpha d) > \bar \varphi(\bar y^k) - \gamma \alpha^2, \quad \forall\ d \in D,
\]
with
$
0 < \alpha \le \epsilon^k.
$
By the mean value theorem,
$\bar \varphi(\bar y^k + \alpha d) - \bar \varphi(\bar y^k) = \alpha \nabla \bar \varphi(\bar y^k + \eta^k \alpha d)^T d$,
for some  $\eta^k \in (0,1)$. Then, for any $k \in K$,
\[
\nabla \bar \varphi(\bar y^k + \eta^k \alpha d)^T d \ge -\gamma \alpha \ge -\gamma \epsilon^k, \quad \forall\ d \in D.
\]
Using the fact that $\eta^k \in (0,1)$, $\alpha \le \epsilon^k$ and $\epsilon^k \to 0$,
we have that
\[
\lim_{\substack{k \to \infty \\ k \in K}} (\bar y^k + \eta^k \alpha d) = y^*, \quad \forall\ d \in D.
\]
Therefore, from the continuity of $\nabla \bar \varphi$ it follows that
\begin{equation}\label{ineq_proof_conv}
\nabla \bar \varphi(y^*)^T d \ge 0, \quad \forall\ d \in D.
\end{equation}
Now consider any point $y \in \Delta_{|\bar {\mathcal A}|-1}$. Reasoning as in the last part of the proof of Theorem~\ref{th:error_stat},
we have that $y-y^* \in T(y^*)$.
Moreover,
%by the same arguments given in the proof of Lemma~\ref{lemma:gen_tan_come} (with minor changes),
it is easy to verify that the set $D^* = \{d \in D \text{ such that $d$ is feasible at $y^*$}\}$ is a set of generators for $T(y^*)$.
Therefore, denoting by $d_1,\ldots,d_{|D^*|}$ the directions that form the set $D^*$, we have that
$y-y^* = \sum_{i=1}^{|D^*|} \lambda_i d_i$, with $\lambda_i \ge 0$, $i=1,\ldots,|D|$.
Taking into account~\eqref{ineq_proof_conv}, it follows that
\[
\nabla \bar \varphi(y^*)^T (y-y^*) = \sum_{i=1}^{|D^*|}\lambda_i \nabla \bar \varphi(y^*)^T d_i \ge 0, \quad \forall\ y \in \Delta_{|\bar {\mathcal A}|-1}.
\]
Then, for all $y \in \Delta_{|\bar {\mathcal A}|-1}$ we have that
\begin{equation*}
\begin{split}
0 & \le \nabla \bar \varphi(y^*)^T (y - y^*) = [\bar A^T \nabla f(\bar A y^*)]^T (y - y^*) = \nabla f(\bar A y^*)^T [\bar A(y - \bar y^*)] \\
                               & = \nabla f(x^*)^T (\bar Ay - x^*).
\end{split}
\end{equation*}
Since $\conv(\bar A) = \{x \in \R^n \colon x = \bar A y, y \in \Delta_{|\mathcal A|-1}\}$,
we obtain that
\[
\nabla f(x^*)^T (x - x^*) \ge 0, \quad \forall\ x \in \conv(\bar A),
\]
implying that~\eqref{stat_proof_1} holds.

To prove~\eqref{stat_proof_2}, note that,
from the instructions of the algorithm, we have that $\mu^{k+1} < \mu^k$ only when the test at line~6 is not satisfied.
Hence, for all $k \in K$,
\[
f(\bar x^k + \mu^k (a - \bar x^k)) > f(\bar x^k) - \gamma (\mu^k)^2, \quad \forall\ a \in {\mathcal A} \setminus \bar {\mathcal A}.
\]
By the mean value theorem, for any $a \in {\mathcal A} \setminus \bar {\mathcal A}$  we can write
\[
f(\bar x^k + \mu^k (a - \bar x^k)) - f(\bar x^k) = \mu^k \nabla f(\bar x^k + \eta^k \mu^k (a - \bar x^k))^T (a - \bar x^k),
\]
for some $\eta^k \in (0,1)$.
Therefore,
\[
\nabla f(\bar x^k + \eta^k \mu^k (a - \bar x^k))^T (a - \bar x^k) > - \gamma \mu^k, \quad \forall\ k \in K.
\]
From Proposition~\ref{prop:lim_hat_mu} and the fact that $\eta^k \in (0,1)$, we have that
\[
\lim_{\substack{k \to \infty \\ k \in K}} (\bar x^k + \eta^k \mu^k (a - \bar x^k)) = x^*.
\]
Therefore, taking into account that $\mu^k \to 0$ and that $\nabla f$ is continuous, we obtain
\[
0 \le \lim_{\substack{k \to \infty \\ k \in K}} \nabla f(\bar x^k + \eta^k \mu^k (a - \bar x^k))^T (a - \bar x^k) = \nabla f(x^*)^T (a - x^*).
\]
Since the above relation holds for all $a \in {\mathcal A} \setminus \bar {\mathcal A}$, we finally get~\eqref{stat_proof_2}.
\end{proof}

\section{Identification property of \ORD}\label{sec:identification}
In our problem, every feasible point is expressed as a (not necessarily unique) convex combination of the atoms $a_i \in \mathcal A$.
In this section we show that, under suitable assumptions, some atoms that are not needed to express the optimal solution
are identified and discarded by \ORD\ in a finite number of iterations.
Loosely speaking, from a certain iteration we are guaranteed that the set ${\mathcal A}^k$ does not contain ``useless'' atoms. Before showing this property, we report a useful intermediate result.
%In particular, the first part of the following proposition says that, even if many different ways might exist to express
%any feasible point of Problem~\eqref{prob} as a convex combination of the atoms $a_1,\ldots,a_n$,
%the weight associated to an atom $a_i$ is zero in every possible convex combination expressing a stationary point $x^*$ if a certain condition holds.

\begin{proposition}\label{prop:atoms_active_set}
Let $x^*$ be a stationary point of Problem~\eqref{prob} and let $w^* \in \Delta_{m-1}$ be any vector such that $x^* = Aw^*$.
Then, for every atom $a_i \in \mathcal A$ such that $\nabla f(x^*)^T (a_i - x^*) > 0$, we have that $w^*_i = 0$.
\end{proposition}

\begin{proof}
Consider the reformulation of Problem~\eqref{prob}  in~\eqref{prob_sim}, with
$
\varphi(w)~=~f(Aw).
$
Let $w^*$ be any feasible point of Problem~\eqref{prob_sim} that $x^* = Aw^*$. Since $x^*$ is stationary for Problem~\eqref{prob}
and $\conv(\mathcal A) = \{x \in \R^n \colon x = A w, w \in \Delta_{m-1}\}$, we have that
\[
\nabla f(x^*)^T (A w - x^*) \ge 0, \quad \forall\ w \in \Delta_{m-1}.
\]
Moreover, for all $w \in \Delta_{m-1}$ we can write
\[
\nabla f(x^*)^T (A w - x^*) = [A^T \nabla f(Aw^*)]^T (w - w^*) = \nabla \varphi(w^*)^T (w - w^*).
\]
It follows that $\nabla \varphi(w^*)^T (w - w^*) \ge 0$ for all $w \in \Delta_{m-1}$, that is, $w^*$ is stationary for Problem~\eqref{prob_sim}
and satisfies the following KKT conditions with multipliers $\lambda^* \in \R$ and $v^* \in \R^m$:
\begin{subequations}
\begin{align}
& \nabla \varphi(w^*) - \lambda^* e - v^* = 0, \label{kkt1} \\
& e^T w^* = 1, \label{kkt2} \\
& (v^*)^T w^* = 0, \label{kkt3} \\
& w^* \ge 0, \label{kkt4} \\
& v^* \ge 0 \label{kkt5}.
\end{align}
\end{subequations}
From~\eqref{kkt1} we can write
\begin{equation}\label{kkt_mult}
v^* = \nabla \varphi(w^*) - \lambda^* e,
\end{equation}
and then, by~\eqref{kkt3} we get that
$0 = (v^*)^T w^* = (\nabla \varphi(w^*) - \lambda^* e)^T w^*$. Using~\eqref{kkt2} we obtain that $\lambda^* = \nabla \varphi(w^*)^T w^*$,
which, combined with~\eqref{kkt_mult}, yields to
\[
v^* = \nabla \varphi(w^*) - (\nabla \varphi(w^*)^T w^*)e
\]
So, for all $h = 1,\ldots,m$ we have that
\begin{equation*}
\begin{split}
v^*_h & = \nabla \varphi(w^*)^T(e_h - w^*) = [A^T \nabla f(Aw^*)]^T (e_h -w^*) = \nabla f(x^*)^T (Ae_h - Aw^*) \\
      & = \nabla f(x^*)^T (a_h - x^*).
\end{split}
\end{equation*}
Therefore, if $\nabla f(x^*)^T (a_i - x^*) > 0$ for an atom $a_i \in \mathcal A$, this means that $v^*_i > 0$ and~\eqref{kkt3}, \eqref{kkt4} and~\eqref{kkt5}
yield to $w^*_i = 0$, thus proving the desired result.
\end{proof}

In the next theorem, we assume that $x^k\to x^*$ (this is pretty standard in the analysis of active-set identification properties) and show that, for $k$ sufficiently large,
the atoms satisfying the condition of Proposition~\ref{prop:atoms_active_set}
are not included in $\mathcal A^k$.
To obtain such a result, we  set ${\mathcal D}^k$ as follows:
\begin{equation}\label{dropcond}
\mathcal D^k = \{a\in \mathcal{A}^k \text{ such that } a=A^k e_h \text{ and } \bar y^k_h=0\}.
\end{equation}
%that is, we remove from $A^k$ all the atoms with zero weight in the current point $\bar y^k$.

\begin{theorem}\label{th:active_set}
Let $\{x^k\}$ be a sequence of points produced by Algorithm~\ref{alg:general}, where ${\mathcal D}^k$ is computed as in~\eqref{dropcond}.
Assume that $\displaystyle{\lim_{k \to \infty} x^k = x^*}$. Then, an iteration $\bar k$ exists such that, for all $k \ge \bar k$,
\[
\nabla f(x^*)^T (a - x^*) > 0,\ a \in \mathcal{A} \; \Rightarrow \; a \notin \mathcal A^k.
\]
\end{theorem}

\begin{proof}
Let $a \in \mathcal{A}$ be an atom such that
\begin{equation}\label{atom_active_set}
\nabla f(x^*)^T(a-x^*)>0.
\end{equation}
First, we want to show that
\begin{equation}\label{a_i_notin_Rk}
a \notin \mathcal R^k, \quad \forall\ \text{ sufficiently large } k.
\end{equation}
Arguing by contradiction, assume that~\eqref{a_i_notin_Rk} is not true.
Then, an infinite subset of iterations $K \subseteq \{0,1\ldots\}$ exists such that $a \in \mathcal R^k$ for all $k \in K$.
From the instructions of the algorithm, we have that
\[
f(\bar x^k + \mu^k (a - \bar x^k)) \le f(\bar x^k) - \gamma (\mu^k)^2, \quad \forall\ k \in K.
\]
By the mean value theorem, we can write
\[
f(\bar x^k + \mu^k (a - \bar x^k)) - f(\bar x^k) = \mu^k \nabla f(\bar x^k + \eta^k \mu^k (a - \bar x^k))^T (a - \bar x^k),
\]
for some $\eta^k \in (0,1)$, and then
\[
\nabla f(\bar x^k + \eta^k \mu^k (a - \bar x^k))^T (a - \bar x^k) \le -\gamma \mu^k, \quad \forall\ k \in K.
\]
From  Corollary~\ref{corol:lim_x} and the fact that $\|\bar x^k - x^*\| \le \|\bar x^k - x^{k+1}\| + \|x^{k+1}-x^*\|$, it follows that $\{\bar x^k\} \to x^*$. Taking also into account that $\eta^k \in (0,1)$ and $\{\mu^k\} \to 0$ (from Proposition~\ref{prop:lim_hat_mu}),
we have that
\[
\lim_{\substack{k \to \infty \\ k \in K}} (\bar x^k + \eta^k \mu^k (a - \bar x^k)) = x^*.
\]
Therefore, using the continuity of $\nabla f$ we obtain
\[
0 \ge \lim_{\substack{k \to \infty \\ k \in K}} \nabla f(\bar x^k + \eta^k \mu^k (a - \bar x^k))^T (a - \bar x^k) = \nabla f(x^*)^T (a - x^*),
\]
which contradicts~\eqref{atom_active_set}. Thus, \eqref{a_i_notin_Rk} holds.

Now, to prove the desired result we proceed by contradiction.
Namely, we assume that an infinite subset of iterations $K \subseteq \{0,1\ldots\}$ exists such that $a \in \mathcal A^k$ for all $k \in K$.
In view of~\eqref{a_i_notin_Rk}, an iteration $\hat k \in K$ must exist such that
\begin{equation}\label{a_i_notin_Dk}
a \in \mathcal A^k \setminus \mathcal D^k, \quad \forall\ k \ge \hat k, \, k \in K.
\end{equation}
Using the fact that $\mathcal A$ is a finite set and the feasible set of every restricted Problem~\eqref{eqprobconv} is compact,
without loss of generality we can assume that $\mathcal A^k$ is constant for all $k \in K$ and that $\{\bar y^k\}$ converges to $y^*$
(passing to a further subsequence if necessary). Namely,
\begin{subequations}\label{subseq_y^k}
\begin{align}
& \mathcal A^k = \bar {\mathcal A}, \quad \forall\ k \in K, \\
& \lim_{\substack{k \to \infty \\ k \in K}} \bar y^k = y^*.
\end{align}
\end{subequations}
Since $\mathcal A^k$ is constant for all $k \in K$, also the matrix $A^k$ and the function $\varphi^k$ are the same for all $k \in K$,
and let us denote them by $\bar A$ and $\bar \varphi$, respectively.
From the previous relations, and taking into account Proposition~\ref{prop:lim_hat_mu}, we also have
\[
x^* =
\lim_{\substack{k \to \infty \\ k \in K}} x^k =
\lim_{\substack{k \to \infty \\ k \in K}} \bar x^k =
\lim_{\substack{k \to \infty \\ k \in K}} \bar A \bar y^k = \bar A y^*.
\]
Moreover, let us denote by $\hat \imath$ the column index of the matrix $\bar A$ the corresponds to the atom $a$, that is, $\bar Ae_{\hat \imath} = a$.

From~\eqref{a_i_notin_Dk} and~\eqref{dropcond}, necessarily $\bar y^k_{\hat \imath} > 0$ for all $k \ge \hat k$, $k \in K$.
Since the set of directions used in \DFSIMPL\ is finite, for all $k \in K$ we can assume that
the directions used in the last iteration of \DFSIMPL\ are the same, having the form
$\pm(e_h-e_{\hat \jmath})$, $h = 1,\ldots,|\bar{\mathcal A}|$, $h \ne \hat \jmath$, for some $\hat \jmath \in \{1,\ldots,|\bar{\mathcal A}|\}$, with
$\bar y^k_{\hat \jmath} > 0$ for all $k \in K$.
In particular, recalling the rule for computing the search directions in \texttt{DF-SIMPLEX} and that the stopping condition~\eqref{stopping_cond_alg_basic} requires that no progress is made along any direction, we have that
\begin{equation}\label{bar_y_k_j}
\bar y^k_{\hat \jmath} \ge \tau/|\bar{\mathcal A}|, \quad \forall\ k \in K.
\end{equation}
Moreover, $e_{\hat \jmath}-e_{\hat \imath}$ is a feasible direction at $\bar y^k$ for all $k \ge \hat k$, since $\bar y^k_{\hat \imath} > 0$.
So, using again the fact that the stopping condition~\eqref{stopping_cond_alg_basic} requires that
no progress is made along any direction,
from the instructions of \DFSIMPL\ we have that
\[
\bar \varphi(\bar y^k + \alpha (e_{\hat \jmath}-e_{\hat \imath})) > \bar \varphi(\bar y^k) - \gamma \alpha^2, \quad k \ge \hat k, \, k \in K,
\]
with $0 < \alpha \le \epsilon^k$.
By the mean value theorem, we can write
\[
\bar \varphi(\bar y^k + \alpha (e_{\hat \jmath}-e_{\hat \imath})) - \bar \varphi(\bar y^k) =
\alpha \nabla \bar \varphi(\bar y^k + \eta^k \alpha (e_{\hat \jmath}-e_{\hat \imath}))^T (e_{\hat \jmath}-e_{\hat \imath}),
\]
for some  $\eta^k \in (0,1)$. Then
\[
\nabla \bar \varphi(\bar y^k + \eta^k \alpha (e_{\hat \jmath}-e_{\hat \imath}))^T (e_{\hat \jmath}-e_{\hat \imath}) \ge -\gamma \alpha, \quad k \ge \hat k, \, k \in K.
\]
Since $\eta^k \in (0,1)$, $\alpha \le \epsilon^k$ and $\{\epsilon^k\} \to 0$,
we have that
\[
\lim_{\substack{k \to \infty \\ k \in K}} (\bar y^k + \eta^k \alpha (e_{\hat \jmath}-e_{\hat \imath})) = y^*.
\]
Therefore, from the continuity of $\nabla \bar \varphi$ and using again the fact that $\{\epsilon^k\} \to 0$, we obtain that
\[
0 \le \nabla \bar \varphi(y^*)^T (e_{\hat \jmath}-e_{\hat \imath}) = [\bar A^T \nabla f(\bar A y^*)]^T (e_{\hat \jmath}-e_{\hat \imath})
  = \nabla f(x^*)^T (\bar Ae_{\hat \jmath}-\bar Ae_{\hat \imath}).
\]
Let us denote by $\tilde a$ the atom that corresponds to the $\hat \jmath$-th column of $\bar A$, that is,
$\bar Ae_{\hat \jmath} = \tilde a$ (also recall that $\bar Ae_{\hat \imath} = a$).
Then
\begin{equation}\label{ineq_proof_active_set}
0 \le \nabla f(x^*)^T (\tilde a-a) = \nabla f(x^*)^T (x^*-a) + \nabla f(x^*)^T (\tilde a-x^*)
\end{equation}
Now, consider the vector $w^* \in \Delta_{m-1}$, obtained from $y^*$ by adding the zero components corresponding to the atoms in
$\mathcal A \setminus \bar {\mathcal A}$, so that $A w^* = \bar A y^* = x^*$.
We can assume, without loss of generality, that $\tilde a$ is also the ${\hat \jmath}$-th column of the full  matrix $A$. Using~\eqref{bar_y_k_j}, we can hence write
$
w^*_{\hat \jmath} > 0.
$
So, from Proposition~\ref{prop:atoms_active_set}
and stationarity of $x^*$, we have that $\nabla f(x^*)^T (\tilde a-x^*) = 0$.
Using this equality in~\eqref{ineq_proof_active_set}, we get $\nabla f(x^*)^T (a-x^*) \le 0$,
thus contradicting~\eqref{atom_active_set}.
\end{proof}

\subsection{Enhancing the \textit{Drop Phase} by gradient estimates}\label{subsec:drop_phase}
Removing from $A^k$ all the atoms with zero weight  might be a too ``aggressive'' strategy (i.e., some of the atoms removed at the first iterations might be useful in the subsequent iterations). Then, we can define a more sophisticated rule to build $\mathcal D^k$
by using approximations of $\nabla\varphi^k(\bar y^k)$.
In particular, at every iteration $k$ we can set
\begin{equation}\label{dropcond2}
\mathcal D^k = \{a\in \mathcal{A}^k \text{ such that } a=A^k e_h,\ \bar y^k_h=0 \text{ and } (g^k)^T(e_h - \bar y^k)\ge 0\},
\end{equation}
where the vector $g^k$ is an approximation of $\nabla \varphi^k(\bar y^k)$ satisfying
\begin{equation}\label{errgcond}
\norm{\nabla \varphi^k(\bar y^k) - g^k} \le r^k,
\end{equation}
with $\{r^k\}$ being a sequence of positive scalars converging to zero
(we will discuss later how to compute $g^k$ efficiently such that~\eqref{errgcond} holds).

The rationale behind this choice lies in the fact that
\begin{equation*}
\nabla \varphi^k(\bar y^k)^T(e_h - \bar y^k)  = [(A^k)^T \nabla f(\bar x^k)]^T (e_h - \bar y^k) %= \nabla f(\bar x^k)^T (A^k e_h - A^k \bar y^k)
= \nabla f(\bar x^k)^T (a - \bar x^k),
\end{equation*}
and then a good approximation of $\nabla \varphi^k(\bar y^k)$ can help us to predict, in a neighborhood of $x^*$, the atoms $a\in \mathcal{A}$ such that
$\nabla f(x^*)^T (a - x^*) > 0$. We now show that this choice of $\mathcal D^k$ ensures the same theoretical properties seen above for~\eqref{dropcond}.

\begin{theorem}\label{th:active_set2}
Let $\{x^k\}$ be a sequence of points produced by Algorithm~\ref{alg:general}, where ${\mathcal D}^k$ is computed as in~\eqref{dropcond2}.
Assume that $\displaystyle{\lim_{k \to \infty} x^k \to x^*}$. Then, an iteration $\bar k$ exists such that, for all $k \ge \bar k$,
\[
\nabla f(x^*)^T (a - x^*) > 0, \ a \in \mathcal{A} \; \Rightarrow \; a \notin \mathcal A^k.
\]
\end{theorem}

\begin{proof}
The first part of the proof is identical to the one given for Theorem~\ref{th:active_set}. Namely, we assume that $a\in \mathcal{A}$ is an atom such that~\eqref{atom_active_set} holds and we obtain~\eqref{a_i_notin_Rk}.
To prove the desired result, we then proceed by contradiction, assuming that an infinite subset of iterations $K \subseteq \{0,1\ldots\}$ exists such that $a \in \mathcal A^k$ for all $k \in K$.
In view of~\eqref{a_i_notin_Rk}, an iteration $\hat k \in K$ must exist such that~\eqref{a_i_notin_Dk} holds.
Now, assuming without loss of generality that $\{y^k\}$ satisfies~\eqref{subseq_y^k}, and using the same definitions of subsequences, matrices and indices given in the proof of Theorem~\ref{th:active_set},
from~\eqref{dropcond2} we have that two possible cases
(that will be shown to lead to a contradiction)
can occur for $k \ge \hat k$, $k \in K$: either
\begin{enumerate*}[label=(\roman*)]
\item $\bar y^k_{\hat \imath} > 0$, or
\item $\bar y^k_{\hat \imath} = 0$ and $(g^k)^T(e_{\hat \imath}-\bar y^k)<0$.
\end{enumerate*}
Since, by the same arguments used in the proof of Theorem~\ref{th:active_set}, the first case cannot occur infinite times,
necessarily $\bar y^k_{\hat \imath} = 0$ and $(g^k)^T(e_{\hat \imath}-\bar y^k)<0$ for all sufficiently large $k \in K$.
Taking into account~\eqref{errgcond}, for all $k \in K$ we can write
\begin{equation*}
\begin{split}
 |\nabla \bar \varphi(\bar y^k)^T(e_{\hat \imath}-\bar y^k)-(g^k)^T(e_{\hat \imath}-\bar y^k)| & =
 |(\nabla \bar \varphi(\bar y^k)-g^k)^T(e_{\hat \imath}-\bar y^k)| \\
 & \le \|\nabla \bar \varphi(\bar y^k)-g^k\|\|e_{\hat \imath}-\bar y^k\| \le \sqrt 2 r^k.
\end{split}
\end{equation*}
Therefore, for all $k \in K$ we have that
\begin{equation*}
\begin{split}
(g^k)^T(e_{\hat \imath}-\bar y^k) & \ge \sqrt 2 r^k + \nabla \bar \varphi(\bar y^k)^T(e_{\hat \imath}-\bar y^k) =
\sqrt 2 r^k + [\bar A^T \nabla f(\bar A \bar y^k)]^T (e_{\hat \imath}-\bar y^k) \\
& = \sqrt 2 r^k + \nabla f(\bar x^k) (\bar A e_{\hat \imath} - \bar A \bar y^k) = \sqrt 2 r^k + \nabla f(\bar x^k) (a - \bar x^k).
\end{split}
\end{equation*}
From the continuity of $\nabla f$ and the fact that $\{r^k\} \to 0$, taking the limits we obtain
\[
\liminf_{\substack{k \to \infty \\ k \in K}} (g^k)^T(e_{\hat \imath}-\bar y^k) \ge \nabla f(x^*)^T (a - \bar x^*) > 0,
\]
leading to a contradiction with the fact that $(g^k)^T(e_{\hat \imath}-\bar y^k)<0$ for all $k \in K$.
\end{proof}

%{\color{red}
Now, we describe how to compute $g^k$ in such a way that condition~\eqref{errgcond} is satisfied.
Since point $\bar y^k$ is obtained in the \textit{Optimize Phase} by running \DFSIMPL\ with a tolerance $\epsilon^k$,
we can simply use the sample points produced in the last iteration of \DFSIMPL\ plus one additional sample point  not belonging to $\Delta_{|\mathcal A^k|-1}$,
that is $\bar y^k - \epsilon_k \frac{\sqrt 2}{|\mathcal A^k|} e$,
to perform a simplex gradient computation in $\R^{|\mathcal A^k|}$ (see, e.g.,~\cite{kelley:1999} for definition of simplex gradient).
The last sample point is needed to have a poised sample set.
More in detail, let $\bar y^k,s^k_1,\ldots,s^k_r$ be all the available sample points, with $r \ge |\mathcal A^k|$, and let us denote $Y^k = \{\bar y^k,s^k_1,\ldots,s^k_r\}$.
Moreover, let
\[
S^k = \begin{bmatrix} s^k_1 - \bar y^k & \ldots & s^k_r - \bar y^k \end{bmatrix}, \quad \text \quad b^k = \begin{bmatrix}\varphi^k(s^k_1)-\varphi^k(\bar y^k) & \ldots & \varphi^k(s^k_r)-\varphi^k(\bar y^k) \end{bmatrix}^T.
\]
We compute $g^k$ as the least-squares solution of $(S^k)^Tg = b^k$.
Under the assumption that $\nabla f$ is Lipschitz continuous with constant $L$,
if the sample set $Y^k$ is poised (i.e., if the columns of $(S^k)^T$ are linearly independent) from Theorem~3.1 in~\cite{custodio:2007} it follows that
$
\norm{\nabla \varphi^k(\bar y^k) - g^k} \le \Bigl( |\mathcal A^k|^{1/2} \frac L2 \bigl\|(\Sigma^k)^{-1}\bigr\| \Bigr) \nu^k,
$
where $\nu^k$ is the radius of the smallest ball centered at $\bar y^k$ enclosing the points $s^k_1,\ldots,s^k_r$, and $\Sigma^k$ is obtained from the reduced singular value decomposition of $S^T/\nu^k$, that is, $S^T/\nu^k = U^k \Sigma^k (V^k)^T$, for proper matrices $U^k$ and $V^k$.

In our case, $\nu^k = \sqrt{2}\epsilon^k$ for all sufficiently large $k$
(it follows from the stopping condition used in \DFSIMPL\ combined with the fact that $\{\epsilon^k\} \to 0$ and the fact that all the directions have norm equal to $\sqrt{2}$). Clearly, $\nu^k \to 0$ as $\epsilon^k \to 0$.
Moreover, it is easy to see that $Y^k$ is poised (it follows from the fact that
\DFSIMPL\ uses directions of the form $\pm (e_i - e_{j_k})$ and we also considered an additional sample point along the direction $-e$).
Using the notion of $\Lambda$-poisedness as given in~\cite{conn:2008a,conn:2008b}, it is also easy to see that $\bigl\|(\Sigma^k)^{-1}\bigr\|$ is upper bounded by a constant $\Lambda$
for all sufficiently large iterations.\footnote{We can identify $\tilde Y^k \subseteq Y^k$, with $|\tilde Y^k| = |\mathcal A^k|$, such that $\tilde Y^k$ is $\tilde \Lambda$-poised in the ball centered at $\bar y^k$ with radius $\nu^k$, and this implies that $Y^k$ is $\Lambda$-poised in the same ball with $\Lambda = |\mathcal A^k|^{1/2}\tilde \Lambda$ (see~\cite{conn:2009}, pag. 63), which, in turn, implies that $Y^k$ is poised
and, from Theorem 2.9 in~\cite{conn:2008b}, that $\bigl\|(\Sigma^k)^{-1}\bigr\| \le |\mathcal A^k|^{1/2} \tilde \Lambda \le m \tilde \Lambda$.}

\section{Numerical experiments}
In this section, we analyze in depth the practical performances of the \ORD\ algorithm.  We carried out all our tests in MATLAB R2020b on an Intel(R) Core(TM) i7-9700 with $16$ GB RAM memory, and used data and performance profiles~\cite{more2009benchmarking} when comparing the method with other algorithms.
Specifically, let $S$ be a set of algorithms and $P$ a set of problems. For each $s\in S$ and $p \in P$, let $t_{p,s}$ be the number of function evaluations required by algorithm $s$ on problem $p$ to satisfy the condition
\begin{equation}\label{eq:stop}
f(x_k) \leq f_L + \tau(f(x_0) - f_L)
\end{equation}
where $0< \tau < 1$ and $f_L$ is the best objective function value achieved by any solver on problem $p$. Then, data and performance profiles of solver $s$ are respectively defined as follows:
 \begin{eqnarray*}
 d_s(\kappa) & = & \frac{1}{|P|}\left|\left\{p\in P: t_{p,s}\leq\kappa(n_p+1)\right\}\right|,\\
 \rho_s(\iota) & = & \frac{1}{|P|}\left|\left\{p\in P: \frac{t_{p,s}}{\min\{t_{p,s'}:s'\in S\}}\leq\iota\right\}\right|,
\end{eqnarray*}
where $n_p$ is the dimension of problem $p$.

\subsection{Preliminary results}
We first chose the following $25$  objective functions from the literature
(see, e.g.,~\cite{andrei2008unconstrained,gould2015cutest}):
Arwhead,
Cosine,
Cube,
Diagonal 8,
Extended Beale,
Extended Cliff,
Extended Denschnb,
Extended Denschnf,
Extended Freudenstein \& Roth,
Extended Hiebert,
Extended Himmelblau,
Extended Maratos,
Extended Penalty,
Extended PSC1,
Extended Rosenbrock,
Extended Trigonometric,
Extended White \& Holst,
Fletchcr,
Genhumps,
Mccormk,
Power,
Quartc,
Sine,
Staircase 1,
Staircase 2.
Then, we built the test problems by randomly generating the atoms with a uniform distribution in~$[0,10]^n$. We would like to highlight that there was no relevant redundancy in the generated atoms. In cases where the atoms in $\cal A$ are highly redundant, it is possible to remove useless atoms by solving a sequence of linear programs. This redundancy test might anyway have a significant computational cost (especially when both the dimension of  the problem and the number of the atoms are large).

In the first experiment, we compared \ORD\ with the following algorithms:
\begin{itemize}
\item \DFSIMPL,  the solver proposed in Section~\ref{sec:df_simplex} for minimization over the unit simplex;
\item \LINCOA~\cite{lincoa}, a trust-region based solver for linearly constrained problems\footnote{We would like to thank Tom M. Ragonneau and Zaikun Zhang for  kindly sharing their MATLAB interface for the \LINCOA\ software.};
\item \NOMAD\ (v3.9.1)~\cite{Nomad,AuLeTr09a}, a solver for non-linearly constrained problems implementing the Mesh Adaptive Direct Search algorithm (MADS);
\item \PSWARM~\cite{vaz2007particle}, a global optimization solver for linearly constrained problems combining pattern search and particle swarm;
\item \SDPEN~\cite{liuzzi2010sequential}, a solver for non-linearly constrained problems based on a sequential penalty approach.
\end{itemize}
When running our tests on \DFSIMPL\ and \LINCOA, we used formulation~\eqref{prob_sim0} to represent the problems
(note that $\bar m = m$ for \DFSIMPL\ in this case).
Since \PSWARM\ and \SDPEN\ only handle inequality constraints, they were run by suitably rewriting~\eqref{prob_sim0} as an inequality constrained problem.
Namely, we used the substitution $y_1 = 1 - \sum_{i=2}^n y_i$ to eliminate the variable $y_1$,
so that the new problem only has the constraints $\sum_{i=2}^n y_i \le 1$ and $y_i \ge 0$, $i=2,\ldots,n$.
%With this reformulation we obtained better results than what we got by replacing the equality constraint by two inequality constraints.

We considered two different versions for \NOMAD. The first one, referred to as \mbox{\NOMAD\ \texttt1}, uses the same formulation as the one used for \PSWARM\ and \SDPEN.
The second one, referred to as \mbox{\NOMAD\ \texttt2}, considers the formulation \eqref{prob} and works in the original space $\R^n$ using a non-quantifiable black box constraint that only indicates if $x$ belongs to $\cal M$ or not (this is carried out by solving a linear program).

We are interested in analyzing the performances of the algorithms for different ratios $m/n$, with  $m$ the number of atoms and $n$ the number of variables. Notice that this might affect the sparsity of the final solution (i.e., the number of atoms needed to assemble $x^*$).
In particular, from Carath\'eodory's theorem~\cite{caratheodory1907variabilitatsbereich} we expect that the larger the ratio $m/n$, the sparser the solution. \ORD\ should hence be more efficient than the competitors for larger values of $m/n$.

So, we fix $n=10$ and set $m \in \{n,5n,10n,20n\}$.
In \ORD\ we stopped the algorithm at the first iteration $k$
that fails the test at line 6 of Algorithm~\ref{alg:general} and such that
\[
\hat \mu^k \le \frac{10^{-4}}{\max_{a_i \in \mathcal A \setminus {\mathcal A}^k}\norm{a_i-\bar x^k}}.
\]
In \DFSIMPL\ we used the stopping condition described in Subsection~\ref{subsec:stopping_cond_alg_basic}, with $\epsilon = 10^{-4}$.
In all the other algorithms, the parameters were set to their default values.
Moreover, we used a budget of $100(n+1)$ function evaluations for every algorithm
and we set the starting point as a randomly chosen vertex of $\Delta_{n-1}$.

\begin{figure}[]
\centering
\subfloat[$n = 10$, $m=10$]
{\includegraphics[scale=0.415, trim = 4.0cm 0cm 5.82cm 0.0cm, clip]{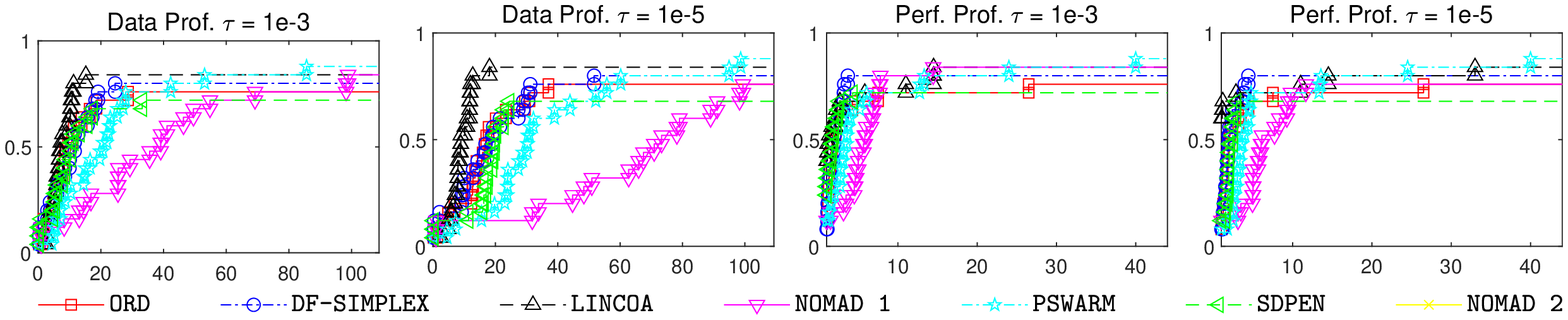}} \\
\subfloat[$n = 10$, $m=50$]
{\includegraphics[scale=0.415, trim = 4.0cm 0cm 5.82cm 0.0cm, clip]{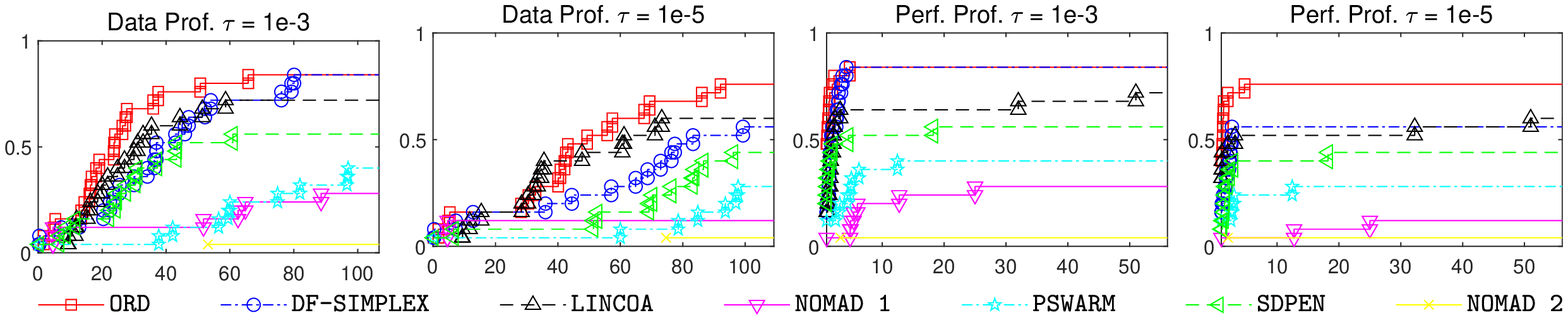}} \\
\subfloat[$n = 10$, $m=100$]
{\includegraphics[scale=0.415, trim = 4.0cm 0cm 5.82cm 0.0cm, clip]{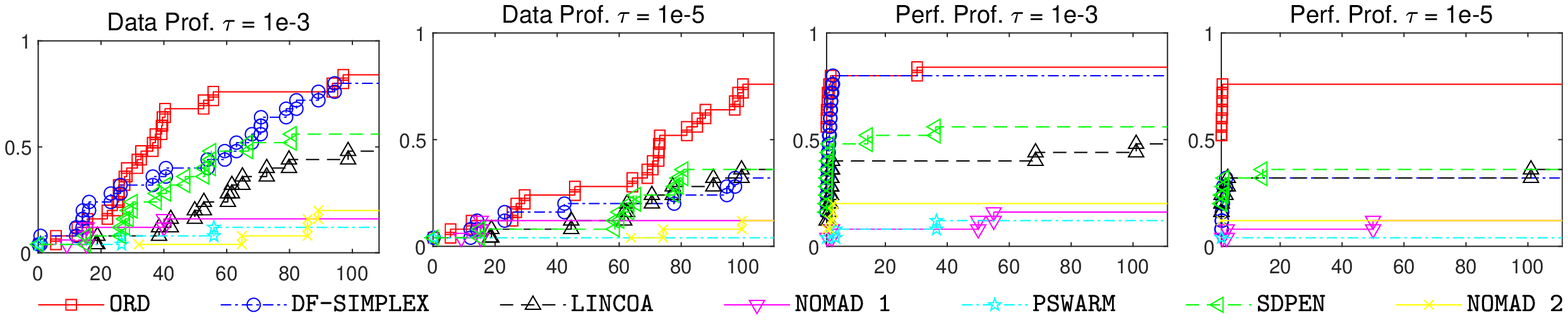}} \\
\subfloat[$n = 10$, $m=200$]
{\includegraphics[scale=0.415, trim = 4.0cm 0cm 5.82cm 0.0cm, clip]{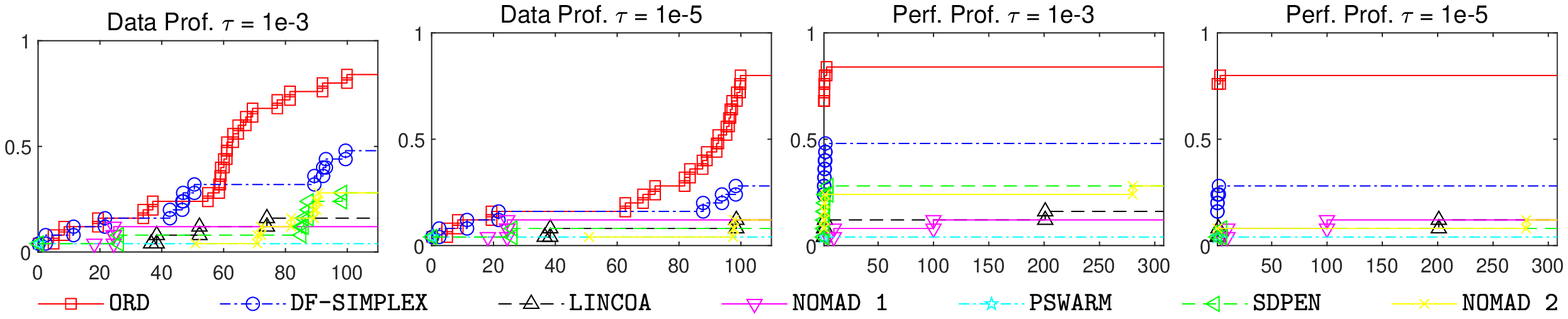}}
\caption{Comparisons among \ORD, \DFSIMPL, \LINCOA, \NOMAD\ \texttt1, \PSWARM, \SDPEN\ and \NOMAD\ \texttt2 for different ratios $m/n$.}
\label{fig:plot_n_10}
\end{figure}

We report, in Figure~\ref{fig:plot_n_10}, the data and performance profiles related to the experiment.
Taking a look at the plots, we see that
\ORD\ clearly outperforms the competitors as the ratio $m/n$ increases (and we get a sparser solution). More specifically, the average sparsity levels
(i.e., the average percentage of atoms with zero weight) of the solutions found by \ORD\
are $62.00\%$ for $m=n$, $87.92\%$ for $m=5n$, $92.68\%$ for $m = 10n$ and $96.08\%$ for $m = 20n$.

In the second experiment, we considered the largest ratio $m/n$, obtained with $m = 20n$, and set the value of $n$ to $20$ and $50$.
For these new experiments, we decided to only run \NOMAD\ \texttt2. There are two main reasons why we did that. First, \NOMAD\ \texttt2
works in the original $n$-dimensional space, while \NOMAD\ \texttt1 works in an $m$-dimensional space (20 times larger than $n$ in these experiments). Second, the maximum number of variables that \NOMAD\ can handle is  $1000$, hence there is no way to run \NOMAD\ \texttt1 on the largest problems anyway.
%
% For these new experiments, we considered only \ORD, \DFSIMPL\ and \LINCOA, the three solvers that got better performances in the previous experiments with $m = 20n$.
%

In Figure~\ref{fig:plot_n_20_50} we report the data and performance profiles related to the new experiment, only including the four solvers that got the best performances, that is, \ORD, \DFSIMPL, \LINCOA\ and \SDPEN.
We see that \ORD\ clearly outperforms the other solvers.
%
%The data and performance profiles related to the new experiment, reported in %Figure~\ref{fig:plot_n_20_50}, show that \ORD\ clearly outperforms the other %solvers.
%
We would also like to notice that the average running time for \ORD\ and \DFSIMPL, both written in MATLAB, is smaller than
$0.1$ seconds for $n=20$ and smaller than $1$ second for $n=50$.
It is the same order of magnitude as \SDPEN, but is much smaller
than \LINCOA, which on average took about $50$ seconds for $n=20$ and about $650$ seconds for $n=50$.
%$1$ second,
%while \LINCOA\ on average took about $50$ seconds for $n=20$ and about $650$ %seconds for $n=50$.

\begin{figure}[]
\centering
\subfloat[$n = 20$, $m=400$]
{\includegraphics[scale=0.415, trim = 4.0cm 0cm 5.9cm 0.0cm, clip]{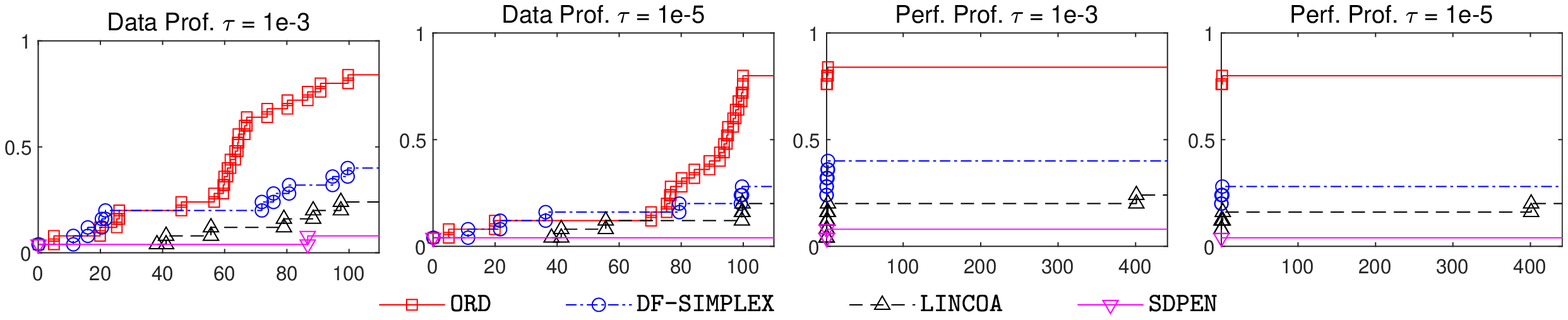}} \\
\subfloat[$n = 50$, $m=1,000$]
{\includegraphics[scale=0.415, trim = 4.0cm 0cm 5.9cm 0.0cm, clip]{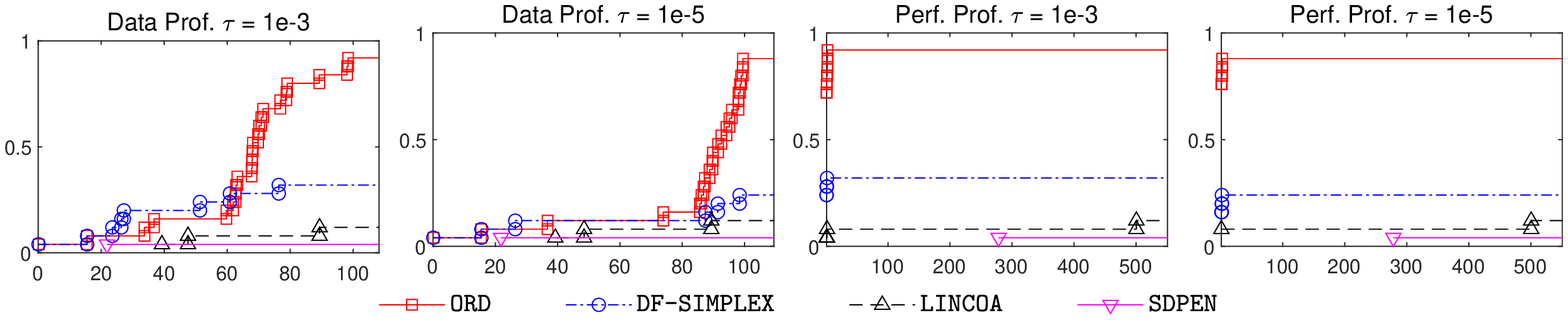}}
\caption{Comparisons among \ORD, \DFSIMPL, \LINCOA\ and \SDPEN\ for different values of $n$ and $m=20n$.}
\label{fig:plot_n_20_50}
\end{figure}

%\subsection{Large-scale instances}
In the final experiment, the aim was to analyze the behavior of the \ORD\ algorithm on relatively large-scale instances. We thus considered once again the largest ratio $m/n = 20$ and set the value of $n$ to $100$, $200$ and $500$. Taking into account the previous results, we  only compared \ORD\ with \DFSIMPL\ in this case.

The data and  performance profiles related to the comparisons, reported in Figure~\ref{fig:plot_n_100_200_500},
confirm once again the effectiveness of \ORD.
In this case, we observed an increased difference between the two considered algorithms in the CPU time required to solve the problem:
\ORD\ on average took about $3$ seconds for $n=100$, about $30$ seconds for $n=200$ and less than $490$ seconds for $n=500$,
while \DFSIMPL\ took less than $1$ second for each problem with $n \in \{100,200\}$ and on average less than $3$ seconds for the problems with $n=500$.
This difference is mainly due to the computation of the simplex gradient that \ORD\ performs in the \textit{Drop Phase}.
Anyway, \ORD\ never exceeded $490$ seconds for solving a problem.

The numerical experiments demonstrate that the methods exploiting the structure of the feasible region (i.e., \ORD, \DFSIMPL\ and \LINCOA) outperform the others. This is not surprising as the latter methods are designed to tackle more general optimization problems.

\begin{figure}[]
\centering
\subfloat[$n = 100$, $m=2,000$]
{\includegraphics[scale=0.415, trim = 4.0cm 0cm 5.9cm 0.0cm, clip]{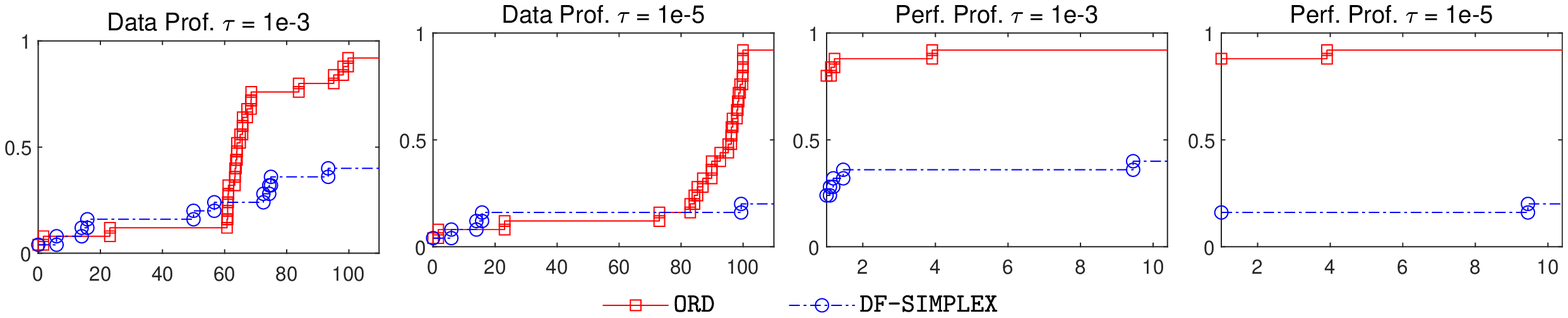}} \\
\subfloat[$n = 200$, $m=4,000$]
{\includegraphics[scale=0.415, trim = 4.0cm 0cm 5.9cm 0.0cm, clip]{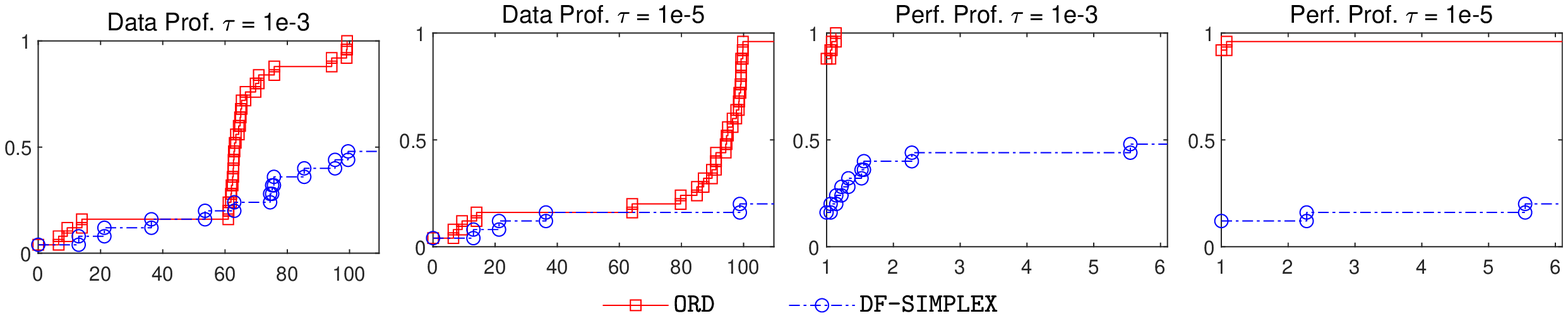}} \\
\subfloat[$n = 500$, $m=10,000$]
{\includegraphics[scale=0.415, trim = 4.0cm 0cm 5.9cm 0.0cm, clip]{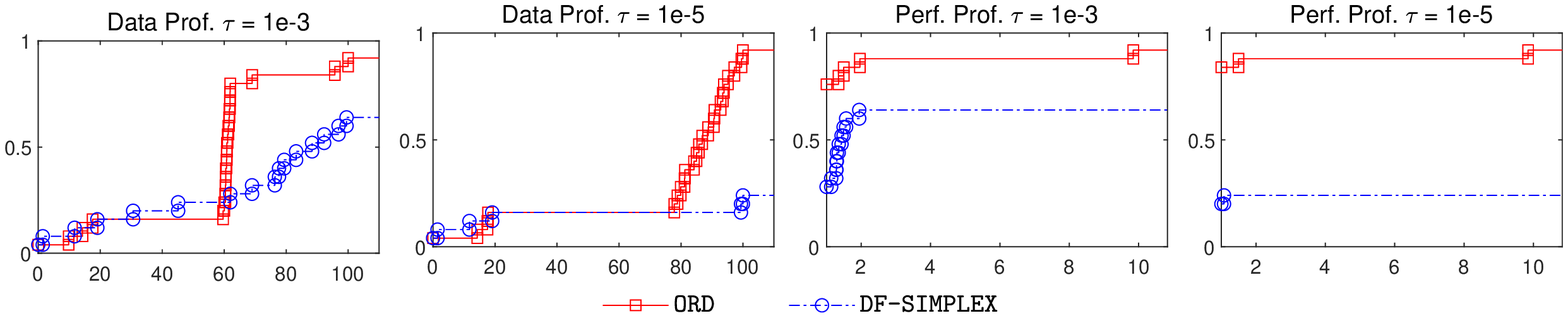}}
\caption{Comparisons between \ORD\ and \DFSIMPL\ on large-scale instances.}
\label{fig:plot_n_100_200_500}
\end{figure}

\subsection{Black-box adversarial machine learning} Adversarial examples are  maliciously perturbed inputs designed to mislead a machine learning model at test time. In  many  fields,  such  as  sign  identification  for  autonomous  driving,  the vulnerability of a model to such examples might have relevant security implications.
An \textit{adversarial attack} hence
consists in taking a correctly classified data point $x_0$ and slightly modifying it to create a new data point that leads a given model to misclassification (see, e.g., \cite{carlini2017towards, chen2017zoo,goodfellow2014generative} for further details).

We now consider a classifier that takes a vector $x \in \R^n$ as an input and outputs $F(x) \in \R^p$, where
$[F(x)]_i \in [0,1]$ represents the \emph{confidence score} for class $i=1,\ldots,p$, i.e., the predicted probability that $x$ belongs to that class, and $\sum_{i=1}^p [F(x)]_i = 1$.

In many real-world applications, the internal configuration of such a classifier is unknown, and one can only access its input and output, i.e., one can only compute $F(x)$.
In this case, we can perform a so-called \emph{black-box adversarial attack} on the model \cite{chen2018frank,chen2017zoo}.

We formulate our problem as a \textit{maximum allowable attack} \cite{chen2018frank,goodfellow2014generative}, namely,
\begin{equation}\label{prob_bb_attack}
\begin{split}
& \min \, f(x_0+x) \\
& s.t. \quad \norm{x}_p \le \varepsilon,
\end{split}
\end{equation}
where $f$ is a suitably chosen attack loss function, $x_0$ is a correctly classified data point, $x$ is the additive noise/perturbation, $\varepsilon > 0$ denotes the magnitude of the attack, and $p\geq 1$.
We set $p=1$ in the formulation \eqref{prob_bb_attack}, thus getting a \emph{maximum allowable $\ell_1$-norm attack}. It is easy to see that  ${\cal M}=\{x\in \mathbb{R}^n:\ \|x\|_1\leq \varepsilon\}=conv({\cal A}),$ with ${\cal A}=\{\pm\varepsilon e_i, i=1,\dots,n\}$, i.e., ${\cal M}$ is a polytope with $2n$ vertices (and then, $m=2n$). This makes the problem fitting our model~\eqref{prob}, and also gets sparsity in the final solution.
We  focus on \emph {untargeted attacks}, i.e., we aim to move
a data point away from its current class, and use the loss function proposed in~\cite{chen2017zoo}:
\begin{equation}\label{bb_loss}
f(z) = \max\{\log[F(z)]_{t_0}-\max_{i \ne t_0} \log[F(z)]_i,-\chi\},
\end{equation}
where $t_0$ is the original class, $\chi$ is a non-negative parameter and $\log 0$ is defined as~$-\infty$.
The rationale behind the use of this loss function is that,
when $\log[F(z)]_{t_0}-\max_{i \ne t_0} \log[F(z)]_i \le 0$,
the sample $z$ is not classified as the original label $t_0$, thus obtaining the desired misclassification. Moreover, the parameter $\chi$ can ensure a gap between
$\log[F(z)]_{t_0}$ and $\max_{i \ne t_0} \log[F(z)]_i$.

In our experiments related to adversarial attacks, we set $\chi=0$ for the loss function~\eqref{bb_loss}, as in~\cite{carlini2017towards,chen2017zoo},
and chose the parameter $\varepsilon$ in problem~\eqref{prob_bb_attack} by means of a parameter selection, using up to $20$ different values.
We obtained $\varepsilon$ values  in the range  $[0.0012n,0.5059n]$.
We thus solved~\eqref{prob_bb_attack} using \ORD\ (our best solver in the preliminary experiments), \LINCOA\ and \SDPEN\ (the best competitors in the preliminary experiments).
Note that an attack is successful only when the objective value is equal to $\chi$, i.e., equal to $0$ in our case. Therefore, in all the algorithms we inhibited any other stopping criterion (that we are allowed to control) and set the maximum number of objective function evaluations equal to $100(n+1)$. Moreover, we set the target objective value equal to $0$ for both \ORD\ and \LINCOA\ (this option is not available for \SDPEN). It is important to notice that for all
the successful attacks we found solutions with a number of non-zero entries smaller than $3\%$.

\subsubsection{Adversarial attacks on binary logistic regression models}
First, we performed untargeted black-box attacks on binary logistic regression models.
We used all the datasets from the LIBSVM web page (\url{https://www.csie.ntu.edu.tw/~cjlin/libsvm/}) with a number of features between $100$ and $2,000$ and a number of training samples less than $50,000$. Here is the complete list: a1a, a2a, a3a, a4a, a5a, a6a, a7a, a8a, a9a, colon-cancer, madelon, mushrooms, w1a, w2a, w3a, w4a, w5a, w6a, w7a and w8a.
%\footnote{\tobe{When
%a split between training and testing data is not available for a given dataset, we use a 90:10  split for the data points (i.e., $90\%$ of the samples in the training set and the remaining $10\%$ in the testing set).}}

We used the training set to build an $\ell_2$-regularized logistic regression model by means of the LIBLINEAR software~\cite{fan2008liblinear} (a built-in cross validation was used to choose the regularization parameter) for all the $20$ datasets.
Then, we randomly selected, for each class and each dataset, a correctly classified test sample $x_0$ and used it in problem~\eqref{prob_bb_attack}, thus getting 40 adversarial attacks.
A built-in LIBLINEAR function was used to compute the probability estimates $[F(x)]_1$ and $[F(x)]_2$ in the loss function~\eqref{bb_loss}.

In Table~\ref{tab:results_lr_attack_1}, we report,  for each solver, the percentage of successful attacks and the average CPU time (in seconds).
We further report, in Figure~\ref{subfig:plot_lr_attack_1}, the percentage of successful attacks versus the required number of simplex gradients.
We see that \ORD\ solves all the problems within a few function evaluations, while \LINCOA\ and \SDPEN\ only solve $77.50\%$ and $30.00\%$ of the problems, respectively. Moreover the CPU time for \ORD\ is always less than $1$ second and, on average, is smaller than \LINCOA\ and \SDPEN\ of $4$ and $2$ orders of magnitude, respectively.

% \begin{table}[htp]
% \centering
% \caption{Adversarial attacks on binary logistic regression models: comparison between \ORD, \LINCOA\ and \SDPEN.}
% \begin{tabular}{|c|D{.}{.}{-1}|D{.}{.}{-1}|}
% \hline
% \textbf{Algorithm} & \multicolumn{1}{c|}{\textbf{Successful attacks}} & \multicolumn{1}{c|}{\textbf{Avg CPU time (s)}}\bigstrut[t] \bigstrut[b] \\
% \hline
% \ORD & 100.00\% & 0.04 \bigstrut[t] \\
% \LINCOA & 77.50\%  & 421.78 \\
% \SDPEN & 30.00\%  & 6.80 \bigstrut[b] \\
% \hline
% \end{tabular}
% \label{tab:results_lr_attack_1}
% \end{table}

\begin{figure}[htp]
\centering
\subfloat[Binary logistic regression models]
{\includegraphics[scale=0.382, trim = 0.2cm 0cm 1cm 0.3cm, clip]{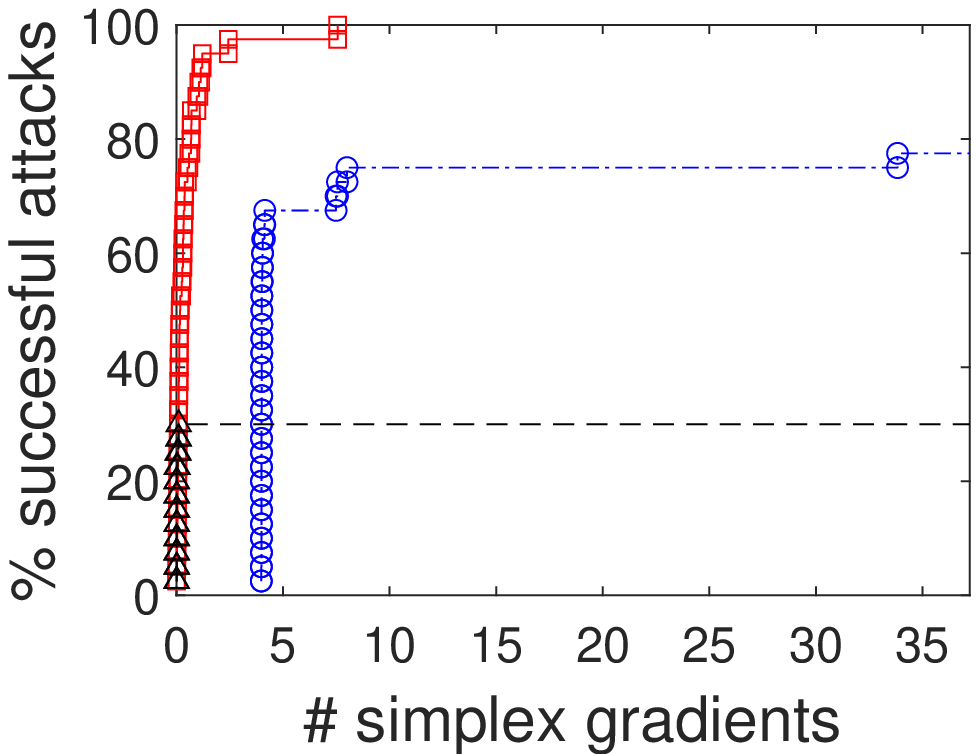}\label{subfig:plot_lr_attack_1}} \;
\subfloat[MATLAB Digits Dataset]
{\includegraphics[scale=0.382, trim = 0.2cm 0cm 1cm 0.3cm, clip]{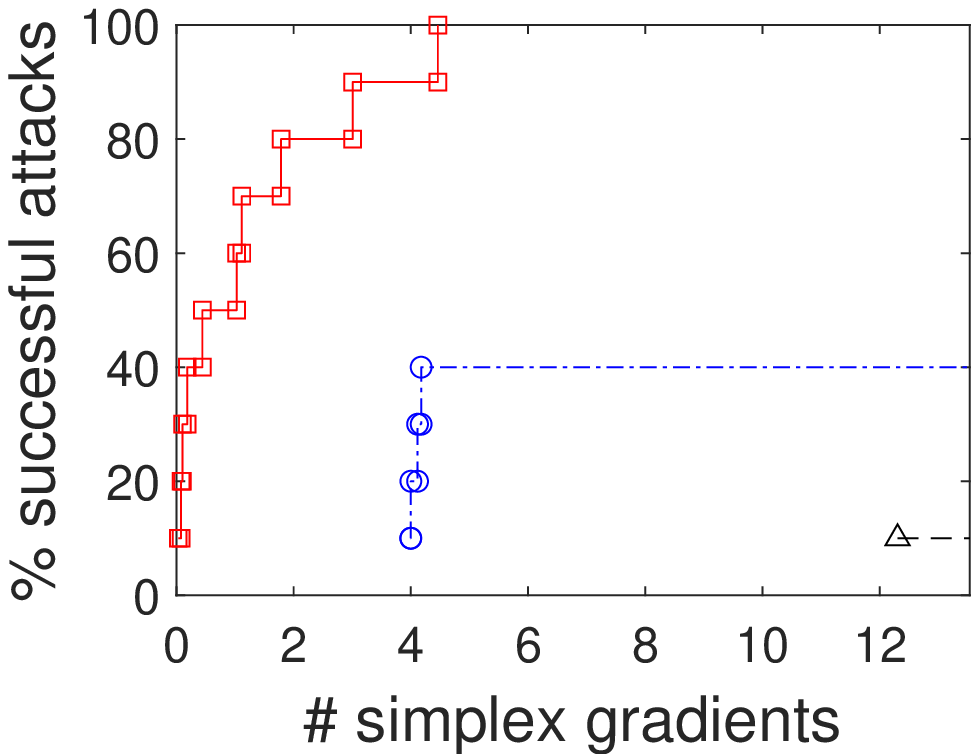}\label{subfig:plot_digits}} \;
\subfloat[Cifar-10 Dataset]
{\includegraphics[scale=0.382, trim = 0.2cm 0cm 1cm 0.3cm, clip]{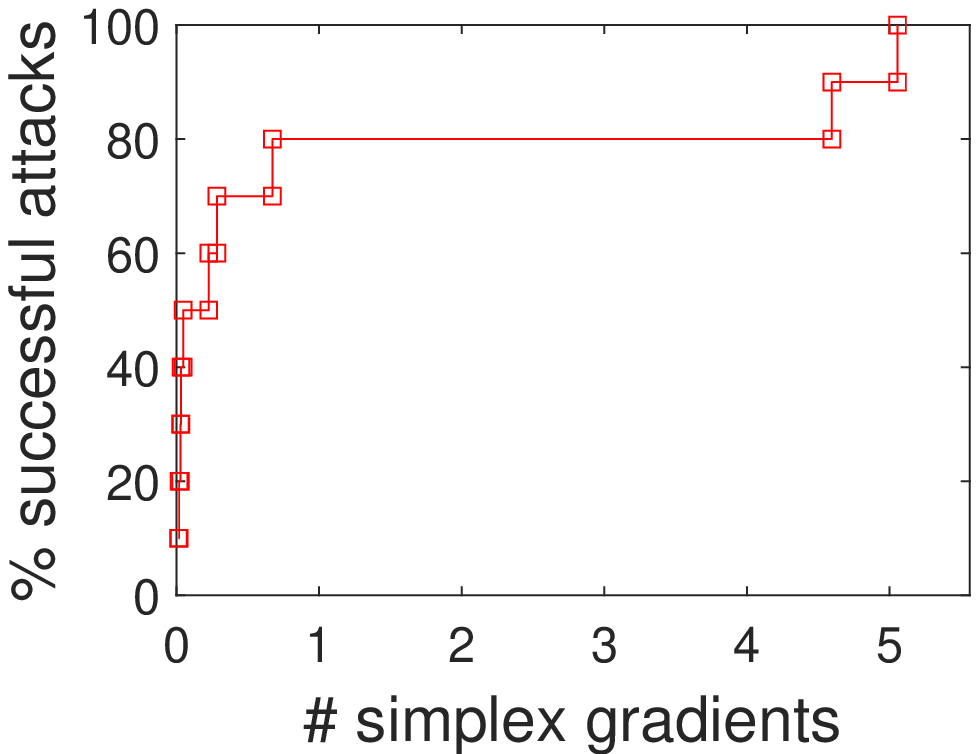}\label{subfig:plot_cifar10}} \;
\subfloat
{\includegraphics[scale=0.36, trim = 4.9cm 0cm 1.85cm 0.0cm, clip]{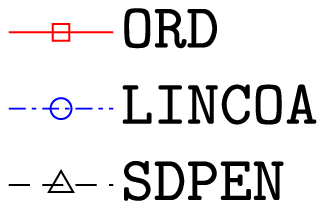}}
\caption{Adversarial attacks on binary logistic regression models (a),
on the MATLAB Digits Dataset  (b), and on the Cifar-10 Dataset (c): percentage of successful attacks vs number of simplex gradients. }
\label{fig:prova}
\end{figure}

\subsubsection{Adversarial attacks on deep neural networks}
In the second experiment, we considered  images of handwritten digits from the MATLAB Digits Dataset. This dataset has $10,000$ $28$-by-$28$ grayscale images of all digits, divided into $10$ classes of $1000$ samples each. The dataset was randomly split using a ratio $90:10$ into  training and testing set. The training set was then used to build a deep neural network with the same architecture as the one described in the examples related to deep learning networks for classification available in MATLAB (see
\texttt{\nolinkurl{https://it.mathworks.com/help/deeplearning/ug/create-simple-deep-learning-network-}}

\noindent
\texttt{for-classification.html} for further details).

We performed untargeted attacks on this deep neural network using  \ORD, \LINCOA\ and \SDPEN\  (notice that $n=784$ and $m=1568$ in this case).
For each class, we randomly selected a correctly classified sample $x_0$ from the validation set and used it in the definition of  problem~\eqref{prob_bb_attack}. Note that each pixel must be a number in the interval $[0,255]$. We hence scaled each variable in the range $[0,1]$. In this case, our formulation \eqref{prob_bb_attack} has a further set of constraints, that is $x_0+x \in [0,1]^n$. In order to get rid of those box constraints, we followed the approach described in~\cite{carlini2017towards,chen2017zoo}, and used a transformation of the form $x_i=(1+\tanh \zeta_i)/2-(x_0)_i$, with $\zeta \in\R^n$.

In Table~\ref{tab:results_dnn_attack_1},  we report the percentage of successful attacks and the average CPU time for each solver.
We further report, in Figure~\ref{subfig:plot_digits}, the percentage of successful attacks versus the required number of simplex gradients.
We see that \ORD\ gets a $100\%$ success rate with an average CPU time of around $1.5$ seconds and a small amount of simplex gradients,
while \LINCOA\ and \SDPEN\ have a success rate lower than $50\%$ and a much larger CPU time.

In Figure~\ref{subfig:digits}, we can see the images obtained with all the attacks applied by \ORD. We notice that the new images are overall very similar to the original ones,  differing, on average, in less than $0.7\%$ of the pixels.

% \begin{table}[htp]
% \centering
% \caption{Results for \ORD, \LINCOA\ and \SDPEN\ on the MATLAB Digits Dataset.}
% \begin{tabular}{|c|D{.}{.}{-1}|D{.}{.}{-1}|}
% \hline
% \textbf{Algorithm} & \multicolumn{1}{c|}{\textbf{Successful attacks}} & \multicolumn{1}{c|}{\textbf{Avg CPU time (s)}}\bigstrut[t] \bigstrut[b] \\
% \hline
% \ORD    & 100.00\% & 1.48 \bigstrut[t] \\
% \LINCOA & 40.00\%  & 1125.93 \\
% \SDPEN  & 10.00\%  & 129.69 \bigstrut[b] \\
% \hline
% \end{tabular}
% \label{tab:results_dnn_attack_1}
% \end{table}

\begin{table}[tbhp]
{\footnotesize
\captionsetup{position=top}
\caption{Adversarial attacks: performance comparison of the DFO methods}
\begin{center}
\subfloat[Binary logistic regression models]{
%\centering
%\caption{}
\begin{tabular}{|c|D{.}{.}{-1}|D{.}{.}{-1}|}
\hline
\textbf{Alg.} & \multicolumn{1}{c|}{\textbf{Success rate}} & \multicolumn{1}{c|}{\textbf{Avg time (s)}}\bigstrut[t] \bigstrut[b] \\
\hline
\ORD & 100.00\% & 0.04 \bigstrut[t] \\
\LINCOA & 77.50\%  & 421.78 \\
\SDPEN & 30.00\%  & 6.80 \bigstrut[b] \\
\hline
\end{tabular}
\label{tab:results_lr_attack_1}
} \quad
\subfloat[MATLAB Digits Dataset]{
%\centering
%\caption{Adversarial attacks on : comparison between \ORD, \LINCOA\ and \SDPEN.}
\begin{tabular}{|c|D{.}{.}{-1}|D{.}{.}{-1}|}
\hline
\textbf{Alg.} & \multicolumn{1}{c|}{\textbf{Success rate}} & \multicolumn{1}{c|}{\textbf{Avg time (s)}}\bigstrut[t] \bigstrut[b] \\
\hline
\ORD    & 100.00\% & 1.48 \bigstrut[t] \\
\LINCOA & 40.00\%  & 1125.93 \\
\SDPEN  & 10.00\%  & 129.69 \bigstrut[b] \\
\hline
\end{tabular}
\label{tab:results_dnn_attack_1}
}
\end{center}
}
\end{table}

\begin{figure}[htp]
\centering
\subfloat[MATLAB Digits Dataset]
{\includegraphics[scale=0.65, trim = 1.4cm 2.6cm 4.2cm 0.7cm, clip]{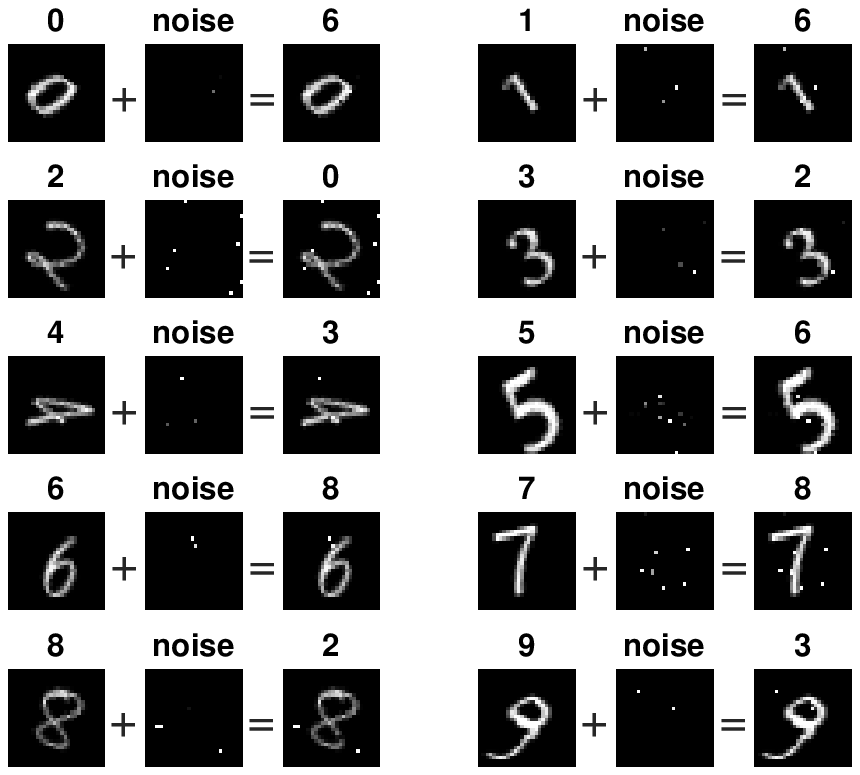}\label{subfig:digits}} \qquad \;
\subfloat[Cifar-10 Dataset]
{\includegraphics[scale=0.65, trim = 1.4cm 2.6cm 4.2cm 0.7cm, clip]{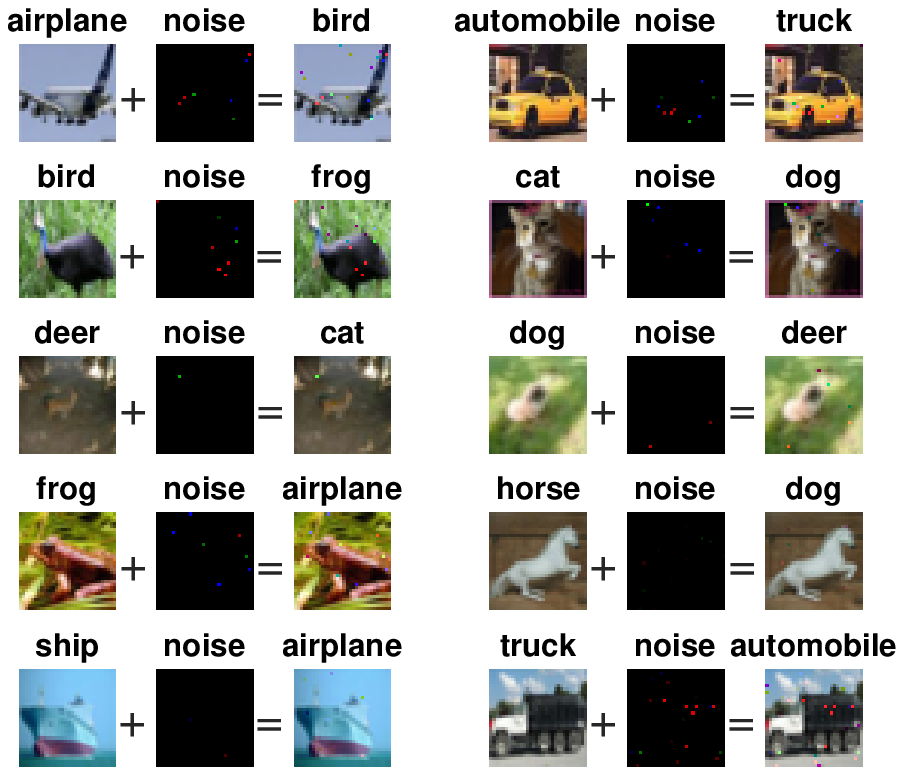}\label{subfig:cifar10}}
\caption{Adversarial attacks applied by \ORD\ on the MATLAB Digits Dataset (a) and on the Cifar-10 Dataset (b).
In each triple, we have on the left the original image with the correct label,
on the right the new image with the misclassified label
and in the middle the additive noise.}
\label{fig:digits}
\end{figure}

Finally, we considered the Cifar-10 dataset~\cite{krizhevsky2009learning} and the trained network described in the MATLAB examples related to the training of residual networks for image classification
(for details see  \url{https://it.mathworks.com/help/deeplearning/ug/train-residual-network-for-image-classification.html}).
The dataset contains $50,000$ samples in the training set and $10,000$ samples in the validation set, where each image is $32$-by-$32$ with $3$ color channels (thus getting  $n=3072$ and $m=6144$).

We performed untargeted adversarial attacks on this deep neural network only using \ORD\ (due to the large dimension of the problems, \LINCOA\ and \SDPEN\, do not give
good results in terms of success rate and/or CPU time). We used the same procedure as the one used for the attacks on the MATLAB Digits Dataset.

We report, in Figure~\ref{subfig:plot_cifar10}, the percentage of successful attacks versus the required number of simplex gradients.
We can see  that \ORD\ achieves $100\%$ success rate (with an average CPU time of slightly more than $30$ seconds) using a few simplex gradients. In Figure~\ref{subfig:cifar10}, we can see the images obtained with all the attacks applied by \ORD.
They are quite similar to the original ones, differing in about $0.2\%$ of the pixels on average.

\section*{Acknowledgments}
The authors would like to thank the two anonymous reviewers for their comments and suggestions that helped to improve the paper.

% \begin{table}[htp]
% \centering
% \caption{Results for \ORD\ on the Cifar-10 Dataset.}
% \begin{tabular}{|c|D{.}{.}{-1}|D{.}{.}{-1}|}
% \hline
% \textbf{Algorithm} & \multicolumn{1}{c|}{\textbf{Successful attacks}} & \multicolumn{1}{c|}{\textbf{Avg CPU time (s)}}\bigstrut[t] \bigstrut[b] \\
% \hline
% \ORD & 100.00\% & 30.99 \bigstrut[t] \bigstrut[b] \\
% \hline
% \end{tabular}
% \label{tab:results_dnn_attack_2}
% \end{table}

\FloatBarrier

\bibliographystyle{siamplain}
\bibliography{references}
\end{document}